\documentclass[oneside, a4paper,12pt]{amsart}

\usepackage{geometry}
\usepackage{amsmath}
\usepackage{amsfonts}
\usepackage{amsthm}
\usepackage{amssymb}
\usepackage{graphicx}
\usepackage{array} 
\usepackage{enumerate}
\usepackage{comment}
\usepackage{hyperref}

\newcommand{\g}{\mathfrak{g}} 	
\renewcommand{\sl}{\mathfrak{sl}}	
\newcommand{\Z}{\mathbb{Z}} 	
\newcommand{\N}{\mathbb{N}} 	
\newcommand{\C}{\mathbb{C}} 	
\newcommand{\acf }{\Bbbk } 	

 \newcommand{\Uplus}[1]{U^+[\overline{#1}]} 
 \newcommand{\Uminus}[1]{U^-[\overline{#1}]}

\newtheoremstyle{introTheorems}
  {\topsep}
  {\topsep}
  {\itshape}
  {0pt}
  {\bfseries}
  {}
  { }
  {\thmname{#1}
  \textnormal{\thmnote{#3}.}
  }

\theoremstyle{plain}
\newtheorem{theorem}{Theorem}[section]

\newtheorem{corollary}[theorem]{Corollary}

\newtheorem{lemma}[theorem]{Lemma}

\newtheorem{proposition}[theorem]{Proposition}
\newtheorem{definition}[theorem]{Definition}
\newtheorem{example}[theorem]{Example}

\newtheorem{problem}[theorem]{Problem}
\newtheorem{fact}[theorem]{Fact}
\newtheorem{question}[theorem]{Question}
\newtheorem{remark}[theorem]{Remark}

\newtheorem*{acknowledgment}{Acknowledgment}

\theoremstyle{introTheorems}

\newtheorem{introExample}{Example}

\newcommand{\gr}{\mathrm{gr}}
\newcommand{\Hom}{\mathrm{Hom}}
\newcommand{\Rep}{\mathrm{Rep}}
\newcommand{\Res}{\mathrm{Res}}

\newcommand{\supp}{\mathrm{supp}}
\newcommand{\rank}{\mathrm{rank}}

\usepackage{keyval}
\makeatletter
\define@key{setpar}{left}[0pt]{\leftmargin=#1}
\define@key{setpar}{right}[0pt]{\rightmargin=#1}
\define@key{setpar}{both}{\leftmargin=#1\relax\rightmargin=#1}
\makeatother
\newenvironment{narrow}[1][]
 {\list{}{\setkeys{setpar}{left,right}%
 \setkeys{setpar}{#1}%
 \listparindent=\parindent
 \topsep=0pt
 \partopsep=0pt
 \parsep=\parskip}\item\relax\hspace*{\listparindent}\ignorespaces}
 {\endlist}

\usepackage{color}

\usepackage[symbol]{footmisc}

\newcommand{\hamburger}[4] 
{
 \thispagestyle{empty}
 \vspace*{-2cm}
 \begin{flushright}
 ZMP-HH #2 \\
 Hamburger Beitr\"age zur Mathematik Nr. #3 \\
 #4 \\
 \end{flushright}
 \vspace{0.5cm}
 \begin{center}
 \Large \bf
 #1
 \end{center}
 \vspace{0.5cm}
 \begin{center} 
 Simon Lentner\\
 Universität Hamburg\\
 Simon.Lentner@uni-hamburg.de\\
 ~\\
 Karolina Vocke\\ 

 University of Innsbruck\\
 karolina.vocke@gmx.de

 \end{center}
 \vspace{0.5cm}
}
\begin{document}

\hamburger{On Borel subalgebras of quantum groups}{17-17}{662}{May 2019}

\vspace{.5cm}
\begin{abstract}
For a quantum group, we study those right coideal subalgebras, for which all irreducible representations are
one-dimensional. If a right coideal subalgebra is maximal with this property, then we call it a \emph{Borel subalgebra}. 

Besides the positive part of the quantum group and its reflections, we find new unfamiliar Borel
subalgebras, for example, ones containing copies of the quantum Weyl algebra. Given a Borel subalgebra, we study its induced (Verma-)modules and prove among others that they have all irreducible finite-dimensional modules as quotients. We give two structural conjectures involving the associated graded right coideal subalgebra, which we prove in certain cases. In particular, they predict the shape of all triangular Borel subalgebras.  
As examples, we determine all Borel subalgebras of $U_q(\sl_2)$ and $U_q(\sl_3)$ and discuss the induced
modules.
\end{abstract}

\makeatletter
\@setabstract
\makeatother
\vspace{4cm}
\begin{acknowledgment}
We are very thankful to Istvan Heckenberger for giving very valuable impulses at several stages of this paper. We are also thankful for his hospitality in Marburg, which was made possible by the Humboldt Institut Partnerschaft. We thank Yorck Sommerhäuser for the comments and literature hints in the final stage of this article, and the anonymous referee for thorough work and many helpful suggestions. Lentner received partial support from the RTG 1670. Part
of this research was conducted while Vocke was visiting Nuffield College, Oxford.
\end{acknowledgment}

\newpage

\thispagestyle{empty}
\tableofcontents

\newpage 

\section{Introduction}

Borel subalgebras are essential objects in the structure theory of a semisimple Lie algebra $\g$, and the representations induced from a Borel subalgebra are essential objects in the representation theory of $\g$. The Borel subalgebras in $\g$ are defined to be the maximally solvable subalgebras, and it is known that all Borel subalgebras are conjugate, in particular isomorphic. By Lie's theorem, all finite-dimensional irreducible representations of a solvable subalgebra are $1$-dimensional. \\ 

Following an analogous question for quantum groups by I. Heckenberger, we want to construct and classify those right coideal subalgebras $C$ of $U_q(\g)$, which have representation-theoretic properties similar to solvable and Borel Lie sub\-algebras: We call a right coideal subalgebra \emph{basic} iff all its finite-dimensional irreducible representations are one-dimensional, 
and we call it \emph{Borel} iff it is maximal with this property. 
It may be surprising that for quantum groups there are additional families of such Borel subalgebras, which are not present in $\g$. An example, already for $\sl_2$ and $q$ not a root of unity, is given later in this introduction.\\

The main motivation for studying Borel subalgebras of quantum groups is the more general goal to understand the set of all right coideal subalgebras, which are also called quantum homogeneous spaces. This is considered to be a main problem in the area of quantum groups, 
with considerable progress made in \cite{Let99, KS08, HS09, HK11a, HK11b}. The main examples are the quantum symmetric pairs, which are constructed in close analogy to the Lie algebra case, starting with \cite{NS95,Let97, Kolb14}.
For Lie algebras, Levi's theorem states that any Lie subalgebra of $\g$ decomposes into a solvable 
Lie algebra and a semisimple Lie algebra. While we have no such decomposition result for general coideal subalgebras, it seems promising to start with a classification of the two classes of coideal subalgebras that correspond to solvable and semisimple Lie subalgebras. In the present article and its predecessor \cite{LV17}, we address the first problem, while \cite{Beck16} addresses the second. \\

\enlargethispage{1cm}
Our second motivation is that the definition of a Borel subalgebra $C$ allows us to repeat the standard construction of Verma modules, namely 
inducing a one-dimensional representation of $C$ to a representation of $U_q(\g)$. 
Our new Borel subalgebras lead to new families of quantum group representations, in which typically the Cartan part acts non-diagonalizable, 
but which otherwise share many features of usual Verma modules, for example having the finite-dimensional irreducible modules as quotients. 
For $\sl_2$, such modules seem to have already appeared as operator-theoretic construction \cite{Schm96} and in Liouville theory \cite[Sec. 19]{Tesch01}. On the other hand, Futorny, Cox and collaborators, starting with \cite{Fut94,Cox94}, have studied non-standard Borel subalgebras of affine Lie algebras; it is conceivable that our construction is related to this construction via Kazhdan-Lusztig correspondence. Third, our non-standard Borel subalgebras suggest non-standard free field realizations in logarithmic conformal field theory \cite[Rem. 2.11]{CLR23}. \\

The goal of the present article is to initiate the systematic study of Borel subalgebras of quantum groups. For one, we construct families of examples and give complete classifications for $\g=\sl_2,\sl_3$. Conceptually, we introduce structural tools and conjecture a general formula for triangular Borel subalgebras. We can prove parts of this conjecture and check it in further examples. Moreover, we prove some general properties of the induced modules and construct them in example classes. \\

As a main structural tool for both questions, we propose to study for a general right coideal subalgebra $C$ the associated graded algebra $\gr(C)$, henceforth referred to as the \emph{graded
algebra}, which is again a coideal subalgebra. The coideal subalgebras $C$ in $U_q^{\geq 0}(\g)$ were classified in \cite{HK11b} as character shifts of graded coideal subalgebras. We propose as \hyperref[conj_A]{Conjecture A$\,$} a conjectural formula for the associated graded coideal subalgebra $\gr(C)$, which may be of independent interest. This also implies a formula for $\gr(C)$ for right coideal subalgebras $C$ in $U_q(\g)$ with a triangular decomposition. We prove some statements that give strong evidence for Conjecture A and prove it in many cases, in particular in all cases for type $A_n$.\\

We expect that this graded algebra $\gr(C)$ essentially determines the representation theory of $C$. From these considerations, we give as \hyperref[conj_B]{Conjecture B$\,$} an explicit formula for all Borel subalgebras with a triangular decomposition. We are able to prove one direction of this conjecture i.e. a necessary condition for a triangular coideal subalgebras to be basic. The other direction would be to prove that the maximal choice of parameters indeed produces a basic right coideal subalgebra, which we can so far only do by explicitly constructing $C$ as an algebra, and to prove that there are no larger non-triangular basic right coideal subalgebras, which we can so far only do in small examples such as $\sl_2,\,\sl_3$. For example, our results show that the basic right coideal subalgebras of $U_q(\sl_{n+1})$ we constructed in \cite{LV17} are maximal at least among all triangular basic coideal subalgebras. For $\sl_2,\,\sl_3$ they are indeed Borel subalgebras, and we give a complete classification.\\



\enlargethispage{1cm}
\noindent
We now discuss the content in more detail: \\

Our main object of study is the quantum group $U=U_q(\g)$, where $\g$ is a finite-dimensional semisimple Lie algebra 
and $q\in\acf^\times$ is not a root of unity. $U$ is a Hopf algebra, which is graded by the root lattice $\Lambda\cong \Z^n$ of $\g$ and has a triangular decomposition $U=U^+U^0U^-$ as algebra, where  $U^0=\acf[\Lambda]$. Our base field $\acf$ is an algebraically closed field
of characteristic $0$. \\

In Section 2, we collect preliminaries on the quantum group $U$ and its
representation theory. A central notion in this article is the concept of a \emph{right coideal subalgebra} $C$, which is a subalgebra of $U$ with coproduct $\Delta(C)\subset C\otimes U$. We introduce in Definition~\ref{def_UwInv} for every element $w$ in the Weyl group $W$ of $\g$ a right coideal subalgebra $\Uminus{w} \subset U^-$ with a Poincare-Birkhoff-Witt (PBW) basis of sorted monomials in root vectors $F_\alpha$, where $\alpha$ runs over a certain subset of positive roots $\Phi^+(w)$. This is a variant of the more common subalgebra $U^-[w]$, which by itself is not a right or left coideal subalgebra, and it is obtained by using Lusztig's automorphisms $T_\alpha^{-1}$ instead of $T_\alpha$ in the definition of root vectors $F_\alpha$.

We then review some major results in the classification of right coideal subalgebras, and translate them to our notation and setting:  In \cite{HS09} it is shown, in a slightly different formulation, that all right coideal subalgebras $C\subset U^{\leq 0}$ that are \emph{homogeneous}, which means $C\supset U^0$ and implies that $C$ is $\Lambda$-graded, are of the form $S(U^-[w])U^0$, which is equal to $\Uminus{w}U^0$. In \cite{HK11b} it is shown, in a slightly different formulation, that all right coideal subalgebras $C\subset U^{\leq 0}$ that are $\Lambda$-graded and are \emph{connected}, which means $C\cap U^0=\acf 1$, are of the form  $\Uminus{w}$. Then it is shown that all connected right coideal subalgebras $C\subset U^{\leq 0}$  are \emph{character shifts} $\Uminus{w}_\phi$, and more generally all coideal subalgebras $C\subset U^{\leq 0}$ with $C\cap U^0=\acf[L]$ a Hopf algebra are of the form $\Uminus{w}_\phi \acf[L]$ for suitable abelian subgroups $L\subset \Lambda$. 

Regarding right coideal subalgebras of the full quantum group $C\subset U$, the results in \cite{HK11a} show, again in our language, that every homogeneous right coideal subalgebra $C\subset U$ has a triangular decomposition $\Uminus{w_-}U^0 S(\Uplus{w_+})$ for suitable combinations of Weyl group elements $w_\pm$. The classification of arbitrary right coideal subalgebras~$C$ with $C\cap U^0$ a Hopf algebra  is still a difficult open problem.\\

In this article we mainly work with right coideal subalgebras with a triangular decomposition  $C=C^{\leq 0}C^{\geq 0}$, where $C^{\leq 0}=C\cap U^{\leq 0}$ and $C^{\geq 0}=C\cap U^{\geq 0}$, and we assume in addition that $C\cap U^0$ is a Hopf algebra.  As an immediate consequence of the classification in \cite{HK11b}, formulated as Corollary \ref{cor_triangular}, both parts $C^{\geq 0},C^{\leq 0}$ of such a triangular right coideal subalgebra  are necessarily character shifts, so 
$$C=\Uminus{w_-}_{\phi_-} \acf [L] S(\Uplus{w_+})_{\phi_+}$$
with unknown compatibility conditions for the data $(w_\pm,\phi_\pm,L)$. Note that many previously studied right coideal subalgebras, such as those associated to quantum symmetric pairs, are themselves not triangular but embed into larger triangular right coideal subalgebras. In fact, during our classification in small rank we found no right coideal subalgebra  that is basic but cannot be embedded into a larger triangular basic right coideal subalgebra.\\

In Section 3, we give our definitions of a \emph{basic} right coideal subalgebra and a \emph{Borel subalgebra}. We discuss thoroughly
two examples: First, we classify all Borel subalgebras of $U_q(\sl_2)$. We find the two standard Borel subalgebras 
$U^{\geq 0}$ and $U^{\leq 0}$, which
are already present in the Lie algebra, and we find a new family
of non-homogeneous Borel subalgebras $B_{\lambda,\lambda'}$:\\

\begin{introExample}[\ref{Weyl algebra}]
In $U_q(\sl_2)$ there is a family of non-homogeneous Borel subalgebras
$B_{\lambda,\lambda'}$ with algebra generators
$$ E^\phi:=EK^{-1}+\lambda K^{-1},\;F^\phi:=F+\lambda'K^{-1},\quad\text{for }
 \lambda\lambda'=\frac{q^2}{(1-q^2)(q-q^{-1})}.$$
They are isomorphic to one another for different values of $\lambda$ and $\lambda'$ by a Hopf automorphism
of $U_q(\g)$ rescaling $E,F$ appropriately. As algebra, they are all  isomorphic to the quantum Weyl algebra
$\langle x,y \rangle / (xy-q^2yx-1)$. The family $B_{\lambda,\lambda'}$ is an open orbit under the automorphism group, and its boundaries are the two graded Borel
subalgebras $U^{\geq 0}$ and $U^{\leq 0}$ for $\lambda\to 0$ and $\lambda'\to 0$, and $B_{\lambda,\lambda'}$ can be thought to interpolate between them for $q\neq 1$.\\
\end{introExample}

As a second initial class of examples, we show in Section \ref{sec_standardBorel}, as expected, but not entirely trivial, that the Borel subalgebras that are homogeneous are precisely the standard Borel subalgebras, i.e. $U^{\geq 0}$ and all its reflections $\Uminus{w}\,U^0\, S(\Uplus{w^{-1}w_0})$.\\

In Section 4, we turn our attention to representation theory: One important application of a Borel
subalgebra $\mathfrak{b}$ of a Lie algebra $\g$ is the construction of
Verma modules $V(\lambda)$, which are induced modules $\g\otimes_{\mathfrak{b}} \acf_\chi$ of a
one-dimensional character on a Borel subalgebra $\chi:\mathfrak{b}\to\acf$,  given by a weight $\lambda\in\mathfrak{h}^*$. 

Analogously, we construct an infinite-dimensional $U$-module $V(B,\chi)$ for any Borel subalgebra $B$ in our sense and any one-dimensional character $\chi$. We then prove some general properties,
in particular, that for any given Borel subalgebra $B$ all finite-dimensional
irreducible representations $L(\lambda)$ of $U$ appear as quotients of 
induced modules of one-dimensional character $\chi$ of $B$.

Then we study exemplary induced modules for the unfamiliar Borel
subalgebras $B_{\lambda,\lambda'}$ in $U_q(\sl_2)$. As an algebra, this is a quantum Weyl algebra, and one-dimensional representations are parameterized by complex numbers $e,f$ with $ef=\lambda\lambda'$. The resulting infinite-dimensional $U_q(\sl_2)$-modules are isomorphic to 
the Cartan part $U^0=\acf[K^{\pm 1}]$, where the $U_q(\sl_2)$-action is in such a way that $K$ acts non-diagonalizable by left-multiplication, while the action of $E,F$ is some expression involving a shift. We determine all sub- and quotient modules. We find that for a generic character $\chi$ the induced module is
irreducible, while for a discrete set of characters $\chi$
there is a unique non-trivial quotient, which is isomorphic to
$L(\lambda)$. \\

In Section 5, we study for a right coideal subalgebra $C$ the associated graded algebra $\gr(C)$, which is again a right coideal subalgebra. In particular, for a triangular right coideal subalgebra of the form above, this amounts to knowing the graded algebra $\gr(S(\Uplus{w_+})_{\phi_+})$, which is simply $S(\Uplus{w_+})$ by our choice of direction for the filtration, and the graded algebra  $\gr(\Uminus{w_-})_{\phi_-})$, which is  more involved. The main goal of this section is to compute the graded algebra $\gr(\Uminus{w})_{\phi})$ depending on the data $(w,\phi)=(w_-,\phi_-)$. Up to a localization issue discussed below, this graded right coideal subalgebra has by \cite{HK11b} again the form 
$$\gr(\Uminus{w}_{\phi})=\Uminus{w'}\acf[L]$$
for some unknown shorter Weyl group element $w'$ and some subgroup $L\subset \Lambda$. Motivated by the degree distributions 
and by quite extensive example calculations, we formulate as \hyperref[conj_A]{Conjecture~A} the general formula 
$$w'=(\prod_{\beta\in \supp(\phi)}s_\beta)w
$$
For example, the previously discussed right coideal subalgebra $B_{\lambda,\lambda'}$ in $U_q(\sl_2)$ has the graded algebra predicted by Conjecture A, with $w=s_1$ and $w'=1$:
$$\gr(B_{\lambda,\lambda'})=\langle K^{-1},  S(E)\rangle
= \Uminus{1}\C[K^{-1}]S(\Uplus{s_1})$$

The remainder of this section is devoted to deriving different criteria that imply Conjecture A and in particular to prove it in type $A_n$ and also in many cases of type $B_n,C_n,D_n$. We now discuss the structure of this section more thoroughly:
\begin{itemize}
 \item In Section 5.1, we formulate Conjecture A and list the consequences for arbitrary coideal subalgebras. 
 We illustrate it by treating the example $A_2$. 
 \item In Section 5.2, we discuss localization: A technical complication is that in general $\gr(C)\cap U^0$ is not a Hopf algebra, just the span of a semigroup $G$ in the root lattice, so \cite{HK11a} does not apply. This is why we formulate our conjecture with respect to the localization $\acf[\tilde{G}]\gr(C)$, where $\tilde{G}$ is the group of fractions of $G$, which still captures for us the essential information.
 \item In Section 5.3, we compare the growth, as a filtered vector space, of $\Uminus{w}$ to $\Uminus{w'}$ to prove the formula $\ell(w)=\rank(\tilde{G})+\ell(w')$. In essence, the formula holds because a suitable notion of growth does not change under going to the localization and to the graded algebra. It would be desirable to obtain this result by comparing Gelfand-Kirillov dimensions (which have similar nice 
 properties), but there are technical issues, and on the other hand  explicit PBW-bases are available, so we use Hilbert series and rely only on the structure as filtered vector spaces, not the algebra structure. 
 \item In Section 5.4, an inductive proof for Conjecture A is set up. Suppose that for some Weyl group element $u$ with $C_u=\Uminus{u}_{\varphi}$ we know that the Weyl group element $u'$ that determines the localization of the graded coideal subalgebra $\gr(C_u)$ is given according to Conjecture A. Then for a longer Weyl group element $w=us_i$  and a right coideal subalgebra $C=\Uminus{w}_\phi$, we find that the Weyl group element $w'$ determining $\gr(C)$ is either of the form $w'=u's_j$ for some $\alpha_j$ or of the form $w'=u'$ depending on the support of the character $\phi$ extending $\varphi$. In the second case Conjecture A holds for $w$ as well, and in the first case Conjecture A holds for $w$ iff $\alpha_i=\alpha_j$.
 \item In Section 5.5, we give two criteria that severely narrow the situations in which this induction step can fail, i.e. $\alpha_i\neq \alpha_j$ contrary to Conjecture~A. Criterion 1 proves that such an alternative $\alpha_j$ can only appear in a narrow list of possibilities, up to reflection, which means specific Weyl group elements for parabolic subsystems of type $A_3,B_3,C_3,D_4$. Criterion 2 proves that  such an alternative $\alpha_j$ can only appear if $u$ has a unique ending for its regular expressions resp. if $\alpha_i\in\Phi^+(u)$ is the only simple root. 
 \item In Section 5.6 we show that these criteria already suffice to prove Conjecture~A in type $A_n$ because in this root system, all roots are multiplicity-free.
 \item In Section 5.7 we conclude with some remarks on the remaining critical cases in type $B_n,C_n,D_n$, where again most cases follow from the criteria above, but a few hard cases remain. 
 \item In Section 5.8 we present an alternative approach, namely the explicit computation of the character shift of a reflection $T_{s_m}^{-1}$, provided $\alpha_m$ appears in the root in question with multiplicity zero.  This calculation uses a slight generalization of \cite{Jan96} 8.14 (6) from the case of a simple root to a root with multiplicity zero. It can be used to compute the graded algebra in explicit cases by hand.
\end{itemize}
Altogether we feel at this point that a more conceptual approach to prove Conjecture A would be much more desirable.\\

In Section 6 we start to apply Conjecture A to control
representation-theoretic properties of a triangular right coideal subalgebra
$C$. We formulate as Conjecture~B a complete description of
the data $(w_\pm,\phi_\pm,L)$ that lead to triangular Borel subalgebras of
$U$, namely precisely those for which $\gr(C)$ localizes to a Borel subalgebra. 
We are only able to prove one direction of our Conjecture B, in cases where Conjecture A holds. Our method to prove that a right coideal subalgebra is not basic is to
find higher-dimensional irreducible $C$-representations as composition
factors in the restriction of a suitable $U$-representation $L(\lambda)$.\\

In Section 7 we thoroughly treat the case $U_q(\sl_3)$, and we give a complete classification of all Borel subalgebras.
We find three types of Borel subalgebras, which can all be given
in terms of generators and relations: 
\begin{itemize}
 \item The standard Borel subalgebras.
 \item Borel subalgebras, that come from the Borel subalgebras
$B_{\lambda,\lambda'}$ in $U_q(\sl_2)$, together with a remaining
standard Borel subalgebra. This is the smallest example of a family we constructed in \cite[Sec. 3]{LV17}.
 \item Another family of Borel subalgebras, that consists of an
 extension of a quantum Weyl algebra by another quantum Weyl algebra.
\end{itemize}
In all these cases we also determine the induced representations.

\section{Preliminaries}

\newcommand{\lightning}{\times}
\newtheorem*{conjectureA}{Conjecture A$\hspace{-.15cm}$}
\newtheorem*{conjectureB}{Conjecture B$\hspace{-.15cm}$}

\subsection{Quantum groups}

Throughout this article $\g$ is a finite-dimensional semisimple Lie algebra of rank $n$. 
Associated to this datum is a root system. 

We denote a fixed set of positive
simple roots by $\Pi=\{\alpha_1,\ldots \alpha_n\}$ and the corresponding
set of all positive
roots by $\Phi^+$. The simple roots are a basis of the root lattice $\Lambda=\Z^n$ with
bilinear form $(-,-)$. In concrete examples, we often use the short hand notation $\alpha_{1^{k_1}\ldots n^{k_n}}:=k_1\alpha_{1}+\cdots + k_n\alpha_n$, for example $\alpha_{112}=\alpha_{1^2 2^1}=2\alpha_1+\alpha_2$. We denote by 
$c_{ij}=2\frac{(\alpha_i,\alpha_j)}{(\alpha_i,\alpha_i)}$ the Cartan matrix and by $W$ the
finite Weyl group acting on the root lattice $\Lambda$, generated by the
simple reflections $s_i$ on the hyperplanes $\alpha_i^\perp$. In concrete examples we sometimes denote by $s_\beta$ the reflection on $\beta^\perp$ for an arbitrary root $\beta$, so in particular $s_{\alpha_i}=s_i$.\\

Let $\acf $ be an algebraically closed field of characteristic $0$, and
let $q\in \acf ^\times$
be not a root of unity. We consider the \emph{quantum group} $U=U_q(\g)$, and for the following standard facts, we refer the reader e.g. to \cite{Jan96}. 
The algebra $U$ is a deformation of the universal enveloping algebra of the
Lie algebra $U(\g)$. It is generated by elements $E_\alpha,F_\alpha$
and $K_\alpha^{\pm 1}$ with $\alpha \in \Pi$. It is graded by $\Lambda$, with $E_\alpha,K_\alpha^{\pm 1},F_\alpha$ in degrees $\alpha,0,-\alpha$, and we denote the homogeneous components by $U_\lambda$ for $\lambda\in \Lambda$. At the same time, this grading describes the adjoint action of $U^0$, in that $K_\alpha$ acts on $U_\lambda$ by~$q^{(\alpha,\lambda)}$.  We
write $[X,Y]_{a }=XY-a YX$ for any $a \in \acf$ and
$X,Y\in U$.

There is a \emph{triangular decomposition} $U=U^+U^0U^-$ into subalgebras $U^+,U^0,U^-$ generated by the $E_\alpha,K_\alpha^{\pm 1}, F_\alpha$ respectively, and we have Hopf subalgebras 
$U^{\geq0}=U^0U^+$ and $U^{\leq0}=U^0U^-$ 
 corresponding to Borel subalgebras of $\g$. The \emph{Cartan part} ${U^0=\C[K_{\alpha_1}^{\pm1},\ldots,K_{\alpha_n}^{\pm1}]}$ is the group ring of the root lattice $\Lambda$.

The algebra $U$ can be endowed with the structure of a Hopf algebra by
\begin{align*}
\Delta(E_\alpha)&=E_\alpha\otimes 1 + K_\alpha \otimes E_\alpha 
&S(E_\alpha)&=-K_\alpha^{-1}E_\alpha\\
\Delta(F_\alpha)&=F_\alpha\otimes K_\alpha^{-1} + 1 \otimes F_\alpha 
&S(F_\alpha)&=-F_\alpha K_\alpha\\
\Delta(K_\alpha)&=K_\alpha\otimes K_\alpha
&S(K_\alpha)&=K_\alpha^{-1}
\end{align*}
This endows the category of
representations $\mathrm{Rep}(U)$ of the quantum group with a tensor product $\otimes_\acf$, which also admits a braiding. As abelian category, $\mathrm{Rep}(U)$ is semisimple and equivalent to $\Rep(\g)$, for $q$ not a root of unity.\\

The following material is covered in \cite[Chp. 8]{Jan96} --- for very different and more systematic proofs for many statements in the context of Nichols algebras see \cite{HS09}.  For every simple root $\alpha_i\in \Pi$ there is an algebra automorphism $T_{\alpha_i}=T_{s_i}$ on $U$ due to Lusztig, which acts on the root lattice $\Lambda$ as simple reflection, concretely $T_{s_i}K_\lambda=K_{s_i\lambda}$. It is a nontrivial result that the maps $T_{s_1},\ldots,T_{s_n}$ obey the braid relations of the Artin braid group associated to the Weyl group $W$ of $\g$. In particular, for $w\in W$ and a reduced expression $w=s_{i_1}\cdots s_{i_t}$ the definition $T_w=T_{s_{i_1}}\cdots T_{s_{i_t}}$ is independent of the choice of a reduced expression for $w$. 
\begin{definition}\label{def_PhiUw}
For a reduced expression $w=s_{i_1}\cdots s_{i_t}$ we define a subset of positive roots $\beta_k$ for $1\leq k\leq t$ by
\begin{align*}
\beta_k&=s_{i_1}\cdots s_{i_{k-1}}\alpha_{i_k} \\
\Phi^+(w)&=\{\beta_1,\beta_2,\ldots,\beta_t\}\\
&=\big\{\beta\in \Phi^+\mid w^{-1}(\beta) <0\big\}
\end{align*}
The second characterization shows that $\Phi^+(w)$ does not depend on the choice of a reduced expression for $w$, however the enumeration $\beta_1,\ldots,\beta_t$ does depend on this choice, for example $\beta_1=\alpha_{i_1}$. This assignment is compatible with the Duflo order 
\begin{align*}
v<w\quad\Leftrightarrow\quad\Phi^+(v)\subset\Phi^+(w)
\end{align*}
We define corresponding \emph{root vectors} in $U^+$
\begin{align*}
E_{\beta_k}&=T_{s_{i_1}}\cdots T_{s_{i_{k-1}}}E_{\alpha_k} \\
U^+[w]&= \mathrm{span}\left\{ E_{\beta_1}^{a_1} \cdots E_{\beta_t}^{a_t} \mid a_1,\ldots,a_t\in\N_0 \right\}
\end{align*}
Again, it can be shown that the space $U^+[w]$ does not depend on the choice of a reduced expression, however the root vectors do depend on the choice, not only their enumeration. Taking the monomials in descending order $E_{\beta_t}^{a_t} \cdots E_{\beta_1}^{a_1}$ is known to span the same space.
\end{definition}
In particular, choosing a reduced expression of the longest element $w_0$ in $W$ gives a PBW-basis of $U^+=U^+[w_0]$. To treat $U^-$,  we use the algebra anti-coalgebra involution $\omega$ defined by 
$$\omega(E_\alpha)=F_\alpha,\;\omega(F_\alpha)=E_\alpha,\;\omega(K_\alpha)=K_\alpha^{-1}$$ for $\alpha\in \Pi$ in \cite[Sec. 4.6]{Jan96}. We can thus define root vectors for $U^-$ as images $\omega(E_{\beta_k})=T_{s_{i_1}}^\omega \cdots T_{s_{i_{k-1}}}^\omega F_{\alpha_k}$, where $T_{s_i}^\omega=\omega T_{s_i} \omega^{-1}$. By \cite[8.14 (9)]{Jan96} $T_{s_i}^\omega$ is proportional to $T_{s_i}$ on any $U_\lambda$, with a scalar factor depending on $\lambda$, so without changing the vector space spanned by them, we may ignore this factor and define:
\begin{align*}
F_{\beta_k}&=T_{s_{i_1}} \cdots T_{s_{i_{k-1}}} F_{\alpha_k} \\
U^-[w]&= \mathrm{span}\left\{ F_{\beta_1}^{a_1} \cdots F_{\beta_t}^{a_t} \mid a_1,\ldots,a_t\in\N_0 \right\}
\end{align*}
As we shall see next, these definitions are not ideal for the study of coideal subalgebras, better suited are  the following variants that use $T_{s_i}^{-1}$ instead of $T_{s_i}$, following \cite[Rem. 2.13]{HK11b}:
\begin{definition}\label{def_UwInv}
In this article we will {\bf exclusively} work with the following variant
\begin{align*}
E_{\beta_k}&=T_{s_{i_1}}^{-1}\cdots T_{s_{i_{k-1}}}^{-1}E_{\alpha_k} \\
\Uplus{w}&= \mathrm{span}\left\{ E_{\beta_1}^{a_1} \cdots E_{\beta_t}^{a_t} \mid a_1,\ldots,a_t\in\N_0 \right\}\\
F_{\beta_k}&=T_{s_{i_1}}^{-1} \cdots T_{s_{i_{k-1}}}^{-1} F_{\alpha_k} \\
\Uminus{w}&= \mathrm{span}\left\{ F_{\beta_1}^{a_1} \cdots F_{\beta_t}^{a_t} \mid a_1,\ldots,a_t\in\N_0 \right\}
\end{align*}
\end{definition}
For example, $U^-[s_1s_2]$ in $U_q(\sl_3)$ is generated by  $F_{\alpha_1}$ and $F_{\alpha_{1}+\alpha_{2}}=T_{s_1}F_{\alpha_2}$ $=F_{\alpha_2}F_{\alpha_1}-q F_{\alpha_1}F_{\alpha_2}$, which is proportional to $F_{\alpha_1}F_{\alpha_2}-q^{-1}F_{\alpha_2}F_{\alpha_1}$, on the other hand  
$\Uminus{s_1s_2}$ is generated by $F_{\alpha_1}$ and $F_{\alpha_{1}+\alpha_{2}}=T_{s_1}^{-1}F_{\alpha_2}=F_{\alpha_1}F_{\alpha_2}-q F_{\alpha_2}F_{\alpha_1}$ 
Since the definition of root vectors already depends on the choice of a reduced expression, we find it acceptable to introduce no new notation for this variant of root vectors. \\

The main objects of interest in this article are the following: 
\begin{definition}
A subalgebra $C$ of a Hopf algebra $H$ is called \emph{right coideal
subalgebra} if $\Delta(C)\subset C\otimes H$. Similarly, a \emph{left coideal subalgebra} fulfills $\Delta(C)\subset H\otimes C$. 
\end{definition}

In order to define right coideal subalgebras explicitly by algebra generators, the following elementary observation is useful: 
\begin{fact}
Let $V$ be a right coideal in $H$, that is, a vector space with $\Delta(V)\subset V\otimes H$, then the subalgebra generated by $V$ is a right coideal subalgebra $C=\langle V \rangle$.
\end{fact}

We now turn our attention to right coideal subalgebras of $U=U_q(\g)$:

\begin{definition}
We call a right coideal subalgebra $C\subset U$ \emph{homogeneous} if $U^0\subset
C$. In particular, $C$ is then homogeneous with respect to the $\Lambda$-grading.
\end{definition}
\begin{example}
Any Hopf subalgebra is in particular a right coideal subalgebra. For example
$U^{\geq 0}$ and $U^{\leq 0}$ are homogeneous right coideal subalgebras. 

The subalgebras $U^-$ and $S(U^+)$
are right coideal subalgebras. On the other hand,  $U^+$ and $S(U^-)$ are not right coideal subalgebras, but  left coideal subalgebra.
\end{example}

A first classification result for right coideal subalgebras of $U$ is the following; see references in the cited source for earlier works on this question: 

\begin{theorem}[\cite{HS09} Theorem 7.3]\label{HS09} For every $w\in W$ there is
	a right coideal subalgebra $U^+[w]U^0=U^0U^+[w]$. Conversely, every homogeneous right coideal subalgebra $C\subset U^{\geq 0}$ is of this form for some $w\in W$. 
\end{theorem}

Next, we consider right coideal subalgebras $C$ of $U^{\leq 0},\,U^{\geq 0}$ that are not necessarily homogeneous in the sense $C\supset U^0$, but still homogeneous with respect to the $\Lambda$-grading, which we will henceforth call in short $\Lambda$-graded. This is for example true for the right coideal subalgebras $U^-,\,S(U^+)$. Here we find that the $\Uminus{w},\Uplus{w}$ in Definition \ref{def_UwInv} are somewhat better behaved then the $U^-[w],U^+[w]$, which are themselves neither right nor left coideal subalgebras. For example, as discussed above, $\Uminus{s_1s_2}$ is generated by $F_{\alpha_1}$ and $F_{\alpha_1+\alpha_2}=T_{s_1}^{-1}F_{\alpha_2}=F_{\alpha_1}F_{\alpha_2}-q F_{\alpha_2}F_{\alpha_1}$ and a quick calculation of their coproducts shows we get a right coideal subalgebra
\begin{align*}\Delta(F_{\alpha_1})
&=1\otimes F_{\alpha_1}
+F_{\alpha_1}\otimes K_{\alpha_1}^{-1}\\
\Delta(F_{\alpha_1+\alpha_2})
&=1\otimes F_{\alpha_1+\alpha_2}
+F_{\alpha_1}\otimes (q^{-1}-q)F_{\alpha_2}K_{\alpha_1}^{-1}
+F_{\alpha_1+\alpha_2}\otimes K_{\alpha_1+\alpha_2}^{-1}
\end{align*}

\begin{theorem}[\cite{HK11b} Theorem 2.12]\label{HK11b_nonshifted}
    The subalgebra $\Uminus{w}$ is a right coideal subalgebra of $U^-$ and $\Uplus{w}$ is a left coideal subalgebra $U^+$. Accordingly, the subalgebra $S(\Uplus{w})$ is a right coideal subalgebra of $S(U^+)\subset U^{\geq 0}$ and $S(\Uminus{w})$ is a left coideal subalgebra of $S(U^-)\subset U^{\leq 0}$. 
    
    Conversely, all right resp. left coideal subalgebras of  $U^{\geq 0}$ resp. $U^{\leq 0}$ that are $\Lambda$-graded and connected, meaning $C\cap U^0=1\acf$, are of this form.
\end{theorem}
\begin{proof}
The result in \cite[Thm. 2.12]{HK11b} is stated somewhat differently, and we now discuss how our statement above is equivalent: First, note that being homogeneous with respect to the $\Lambda$-grading is equivalent to the assumption of $\mathrm{ad}(U^0)$-stability in the cited source, for $q$ not a root of unity. Then, the cited source relies on an explicit algebra isomorphism $\psi:U^+\to S(U^+)$, and in  \cite[Lm. 2.11]{HK11b} it is proven that $\psi(U^+[w])=S(U^+)\cap U^+[w]U^0$, which is an  intersection of right coideal subalgebras. Applying the anti-algebra and anti-coalgebra map $S^{-1}$ equivalently describes a left coideal subalgebra  $S^{-1}\psi(U^+[w])=U^+\cap S^{-1}(U^+[w])U^0$ of $U^+$. Now $S^{-1}\psi$ is an anti-algebra isomorphism of $U^+$, which acts on generators by $E_{\alpha_i}\mapsto -E_{\alpha_i}$ and thus agrees, up to the sign depending on the $\Lambda$-degree, with the anti-algebra isomorphism $\tau:U^+\to U^+$ in \cite[Sec. 4.6]{Jan96}, which is on the generators $E_{\alpha_i}$ the identity. Since $T_{s_i}^{-1}=\tau T_{s_i} \tau^{-1}$ by  \cite[ 8.14 (10)]{Jan96}, this proves 
$$\Uplus{{w}}=S^{-1}\psi(U^+[w])=U^+\cap S^{-1}(U^+[w])U^0$$
and it is a left coideal subalgebra of $U^+$. Comparing with the result \ref{HS09}, we have in particular 
$$S(\Uplus{{w}})U^0=U^+[w]U^0$$
The other three versions of the assertion follow easily by applying $\omega$ and $S$.
\end{proof}

Next, we consider right coideal subalgebras in $U^{\geq 0}$ that are not necessarily $\Lambda$-graded anymore, and again we slightly rewrite the result in our notation:

\begin{theorem}[\cite{HK11b} Theorem 2.17]\label{HK11b_shifted}
For $w\in W$, let $\phi:\,\Uminus{w}\to \acf $ be a one-dimensional character and define 
$$\supp(\phi):=\{\beta\in \Lambda\mid \exists x_\beta\in \Uminus{w}_\beta
\text{ with }\phi(x_\beta)\neq 0\},$$
which consists of mutually orthogonal roots. Then there is a connected right coideal subalgebra called \emph{character shift}
$$\Uminus{w}_{\phi}:=\left\{\phi(x^{(1)})x^{(2)}\mid x\in \Uminus{w}\right\}$$
More generally, for every subgroup $L\subset \supp(\phi)^{\perp}$ there is a right coideal subalgebra $C=\Uminus{w}_\phi \acf[L]$ with $C\cap U^0=\acf[L]$. 
Conversely, every right coideal subalgebra $C\subset U^{\leq 0}$ with $C\cap U^0$ a Hopf algebra is of this form for some datum $(w,\phi,L)$. 
\end{theorem}
 Similarly, for every one-dimensional character $\phi:\,S(\Uplus{w})\to \acf$ there is a connected right coideal subalgebra 
$$S(\Uplus{w})_{\phi}:=\left\{\phi(x^{(1)})x^{(2)}\mid x\in S(\Uplus{w})\right\}$$

Next we discuss right coideal subalgebras of the full quantum group $U$. We again start with the case where $C$ is homogeneous, meaning $C\supset U^0$: 

\begin{theorem}[\cite{HK11a} Corollary 3.9]\label{HK11a}
Every homogeneous right coideal subalgebra $C$ of $U$ is of the form 
$$C=S(U^-[w_-])\, U^0\,U^+[w_+]$$
for a certain subset of pairs $w_-,w_+\in W$. 
\end{theorem}
Note that while this formulation of this statement matches Theorem \ref{HS09}, the same result can again be formulated with the $\Uminus{w_+},\Uplus{w_-}$, which has the advantage of presenting $C$ as a product of right coideal subalgebras:
$$C=\Uminus{w_-}\, U^0 \, S(\Uplus{w_+})$$

An interesting and probably very hard question is the classification of arbitrary right coideal subalgebras $C$ of $U$, and one may again assume that $C\cap U^0$ is a Hopf subalgebra, i.e. spanned by a subgroup of $\Lambda$ instead of just a subsemigroup.\\

Note that the previous results, in the form we presented them, give immediately a necessary form of a particular class of right coideal subalgebras:

\begin{definition}\label{def_triangular}
We call a right coideal subalgebra $C$ of $U$ \emph{triangular}, if
$$C=C^{\leq 0}C^{\geq 0},\qquad
C^{\leq 0}=C\cap U^{\leq 0},\qquad \text{where }
C^{\geq 0}=C\cap U^{\geq 0}
$$
We further denote $C^0=C\cap U^0$.
\end{definition} 

\begin{corollary}\label{cor_triangular}
Let $C$ be a triangular right coideal subalgebra in $U$ and assume that $C^0$ is a Hopf algebra. Then it is necessarily of the form 
$$C=\Uminus{w_-}_{\phi_-} \, \acf[L] \, S(\Uplus{w_+})_{\phi_+}$$
for Weyl group elements $w_-,w_+\in W$ and one-dimensional characters $\phi_-,\phi_+$ and a subgroup  $L\subset \supp(\phi_-)^{\perp}\cap \supp(\phi_+)^{\perp}$. 
\end{corollary}
\begin{question}
For which data $(w_-,w_+,\phi_-,\phi_+,L)$ is this product of right coideal subalgebras, which is a right coideal by definition, also multiplicatively closed and hence indeed defines a coideal subalgebra?
\end{question}

\subsection{Quantum group representation theory}

A main idea for Lie algebra and quantum group representations is to take a
large commutative subalgebra $\mathfrak{h}\subset \mathfrak{g}$ (Cartan part)
and simultaneously diagonalize its action. The possible eigenvalues
$\mathfrak{h}\to \acf $ are called weights.

\begin{definition}[Verma module]\label{defverma}
Every one-dimensional character $\chi$ of
$U^0=\acf [\Lambda]$ can be
extended trivially to $U^+U^0$. Consider the induced $U$-representation 
$$V(\chi)=U \otimes_{U^+U^0} \acf _\chi \cong U^-v_\lambda$$
generated by a \emph{highest weight vector} $v_\lambda:=1\otimes 1_\chi$ with 
 $$E_{\alpha}v_{\lambda}=0\text{ and }K_{\alpha}v_{\lambda}=\chi(K_\alpha)v_{\lambda}\text{ for all }\alpha\in\Pi.$$
The module $V(\chi)$ has a unique irreducible quotient module $L(\chi)$.\\

In particular, for a weight $\lambda$ of $\g$, the Verma module $V(\lambda)$ of type $+1$ is the induced representation associated to the character $\chi(K_\alpha)=q^{(\lambda,\alpha)}$. If $\lambda$ is an \emph{integral dominant weight} then $L(\lambda)$ is finite-dimensional. Every finite-dimensional irreducible module is the quotient of the induced module for a unique character $\chi$. 

\end{definition}

Recall that we always assume $q$ is not a root of unity. In this case, the category of finite-dimensional $U$-modules is semisimple and resembles the case $U(\g)$. 

\begin{example}
 For every $n\geq0$ there are two irreducible $U_q(\sl_2)$-modules $L(n,\pm)$ of dimension $n+1$ with basis $m_0,m_1,\ldots, m_n$ such that 

 \begin{minipage}{0.5\textwidth}
\begin{align*}&Km_i=\pm q^{n-2i}m_i\\
 &Fm_i=\begin{cases}
 m_{i+1},& \text{ for } i<n,\\
 0,& \text{ for } i=n
 \end{cases}\\
 &Em_i=\begin{cases}
 \pm [i]_q[n+1-i]_q m_{i-1},& \text{ for } i>0,\\
 0,& \text{ for } i=0
 \end{cases}
 \end{align*}
\end{minipage}

\noindent
Every finite-dimensional irreducible $U_q(\sl_2)$-module is of this form. 
\end{example}

\section{Borel subalgebras}

\subsection{Main Definitions}

\noindent
We are interested in a particular class of right coideal subalgebras:

\begin{definition}
We call an algebra \emph{basic}
if all its finite-dimensional irreducible representations are one-dimensional.
\end{definition}

\begin{example}
Finite dimensional algebras are basic iff they
are basic in the usual sense i.e. $A/\mathrm{Rad}(A)\cong \acf^n$.
Commutative algebras are examples of basic algebras. Universal enveloping algebras
of solvable Lie algebras are basic. 
\end{example}

\begin{definition}\label{def_Borel}
We call a right coideal subalgebra of the Hopf algebra $U=U_q(\g)$ a \emph{Borel
subalgebra} if it is basic and it is maximal (with respect to inclusion) among
all basic right coideal subalgebras.
\end{definition}

The idea of this notion is to generalize the characterization of a Borel subalgebra as the maximal solvable Lie subalgebra.

\begin{example}\label{standardborel}
All homogeneous Borel subalgebras $B$ of $U$, which we recall means $B\supset U^0$, are isomorphic
to the \emph{standard Borel subalgebra} $U^{\geq 0}$, via some Lusztig automorphism $T_w$.
We will prove this in Section~\ref{sec_standardBorel}.
\end{example} 

\begin{example}\label{Weyl algebra}
In $U_q(\sl_2)$ there is a family of non-homogeneous and non-$\Lambda$-graded triangular Borel subalgebras
$B_{\lambda,\lambda'}$ with algebra generators
$$E^\phi:=EK^{-1}+\lambda K^{-1},\qquad F^\phi:=F+\lambda'K^{-1},\qquad
\lambda\lambda'=\frac{q^2}{(1-q^2)(q-q^{-1})}$$
We can check directly that the elements $E^\phi,F^\chi$ each generate a right coideal subalgebra, independent on the condition on $\lambda\lambda'$, because the coproducts are
$$\Delta(E^\phi)=1\otimes EK^{-1}+E^\phi\otimes K^ {-1},\qquad
\Delta(F^\phi)=1\otimes F+F^\phi\otimes K^ {-1},\qquad
$$
Hence together they again generate a right coideal subalgebra. The condition on $\lambda\lambda'$ enters when we want to determine which subalgebra the elements $E^\phi,F^\phi$ generate. We compute the commutator 
\begin{align*}
[E^\phi,F^\phi]_{q^2}&=[EK^{-1}+\lambda K^{-1},\;F+\lambda'K^{-1}]_{q^2}\\
&=\frac{K-K^{-1}}{q-q^{-1}}q^2K^{-1}+\lambda\lambda' K^{-1}K^{-1}(1-q^2)\\
&=\frac{q^2}{q-q^{-1}}\cdot 1
\end{align*}
so the algebra generated by $E^\phi,F^\phi$ is isomorphic to the quantum Weyl algebra, which is basic, see Lemma \ref{lm_qWeyl_basic} below. The maximality among all basic right coideal subalgebra will be proven in the next section.  On the other hand, if the value of $\lambda\lambda'$ is not as stated, then this commutator also contains $K^{-2}$ as a summand, and since $1,K^{-2}$ commute differently with $E^\phi,F^\phi$ we ultimately generate an algebra close to the entire algebra $U_q(\sl_2)$, in particular it cannot be basic.\\

In the setting of Corollary \ref{cor_triangular}, this triangular right coideal subalgebra is 
$$B_{\lambda,\lambda'}=\Uminus{s_1}_{\phi_-}\,\acf[1]\, S(\Uplus{s_1})_{\phi_+}$$
for characters $\phi_+(EK^{-1})=\lambda$ and $\phi_-(F)=\lambda'$, under the stated condition on $\lambda\lambda'$.\\

We also remark that all the $B_{\lambda,\lambda'}$ for different compatible choices of $\lambda,\lambda'$ are isomorphic under a Hopf algebra automorphism of $U$ given by a rescaling of $E,F$. The family $B_{\lambda,\lambda'}$ is an open orbit under this action, and the boundary of this orbit consists of the graded Borel subalgebras $U^{\geq0}$ for $\lambda'\to\infty$ and $U^{\leq0}$ for
$\lambda\to\infty$. We can think of $B_{\lambda,\lambda'}$ as interpolating between them. Due to the denominator there seem to be no analogs of these Borel subalgebras in the classical limit $q\to 1$.
\end{example}

We now prove that $B_{\lambda,\lambda'}$ for $\lambda\lambda'=\frac{q^2}{(1-q^2)(q-q^{-1})}$ is basic, because it is isomorphic to the quantum Weyl algebra. Note that to satisfy conventions we replaced $q$ by $q^2$ and rescale $E^\phi$ or $F^\phi$ to remove the scalar $\frac{q^2}{q-q^{-1}}$ in the commutator $[E^\phi,F^\phi]_{q^2}$:

\begin{lemma}\label{lm_qWeyl_basic}
The \emph{quantum Weyl algebra}, defined by
$$\langle X,Y\rangle/(XY-qYX-1)$$ 
is basic, if $q$ is not a root of unity.
\end{lemma}
\begin{proof}
Let $V$ be a finite-dimensional irreducible representation. Consider an eigenvector $v$ of the element $T:=YX$ with eigenvalue $t$ (which is related to the quantum Casimir). One can easily see that $Yv$ is an eigenvector (or zero) of $T$ with eigenvalue $qt+1$ (or zero), since
$$YX(Yv)=Y(qYX+1)v=(qt+1)Yv$$
Similarly, one can show that $Xv$ is an eigenvector (or zero) of $T$ with eigenvalue $\frac{1}{q}(t-1)$, since
$$YX(Xv)=\frac{1}{q}(XYX-X)v=\frac{1}{q}(t-1)Xv$$
and so forth. As $V$ is finite-dimensional, there cannot be infinitely many pairwise distinct eigenvalues. So either we have in this chain a nonzero vector $v'$ with  $Xv'=0,\,Y^kv'=0$, but then applying $X$ to $Y^kv'=0$ gives $\frac{q^k-1}{q-1}v'=0$, but $q$ is not a root of unity. Or the only appearing eigenvalue of $T$ has to be the unique fixed-point $t=\frac{1}{1-q}$ of the recursion formula $t\mapsto qt+1$. For this eigenvalue $t$ we get
$$XYv=(qYX +1)v=\left(\frac{1}{1-q}q+1\right)v=\frac{(q+1-q)}{1-q}v=tv$$
As $T$ has only the eigenvalue $t$, it acts as a scalar on $V$, the same is true for $XY$. 
Thus $X$ and $Y$ commute on all of $V$, since $V$ is assumed irreducible, and $V$ factors to a representation of the commutative algebra $\acf[X,Y]/(XY-t)$. In particular, any finite-dimensional irreducible representation is one-dimensional, as asserted.
\end{proof}


\subsection{Example \texorpdfstring{$A_1$}{A1}}\label{sec_exampleA1}
We now classify all Borel subalgebras of $U_q(\sl_2)$.
The root system $A_1$ of rank $1$ has $\Phi=\{\alpha,-\alpha\}$ with ${(\alpha,\alpha)=2}$. 
The corresponding quantum group
$U_q(\sl_2)$ is generated by $E,F,K,K^{-1}$ such that
$$[E,F]_1=
\frac{K-K^{-1}}{q-q^{-1}}\quad [E,K]_{q^{-2}}=[F,K]_{q^2}=0$$

\begin{theorem}\label{thm_BorelKlassifikation_sl2}
The Borel subalgebras of $U_q(\sl_2)$ are 
\begin{itemize}
 \item The standard Borel subalgebras $U^{\geq0}$ and $U^{\leq0}$.
 \item For any pair of scalars with $\lambda\lambda'=\frac{q^2}{(1-q^2)(q-q^{-1})}$ the algebra in Example \ref{Weyl algebra}
 $$B_{\lambda,\lambda'}:=\langle EK^{-1}+\lambda K^{-1}, F+\lambda'K^{-1}\rangle$$
\end{itemize}
\end{theorem}
\begin{proof}
We know from the previous lemma that all algebras in the assertion are basic and we know (or check immediately) that they
are right coideal subalgebras. The classification and the maximality are more difficult issues:

We know from \cite[Thm 4.11, Lm 3.5]{Vocke16} that any right coideal
subalgebra in $U_q(\sl_2)$ has a set of generators of the form
$K^i,\;EK^{-1}+\lambda_EK^{-1},\;F+\lambda_FK^{-1},\; EK^{-1}+c_FF+c_KK^{-1}$
for some constants $c_F,c_K,\lambda_F,\lambda_E$. We now check for all
combinations of two such elements, which algebra they roughly generate:

\begin{narrow}[left=-1cm]
\begin{center}
\begin{tabular}{|c||c|c|c|c|}
\hline
 & $K$ & $EK^{-1}+\lambda_E'K^{-1}$ & $F+\lambda_F'K^{-1}$ &$EK^{-1}+c_F'F+c_K'K^{-1}$\\
 \hline\hline
 $K$ & $U^0$ & $U^{\geq0}$ & $U^{\leq0}$ & $U$\\
 \hline
 $EK^{-1}+\lambda_EK^{-1}$ & $U^{\geq0}$ & $\subseteq U^{\geq0}$ & $B_{\lambda,\lambda'}$& $B_{\lambda,\lambda'}$ \\
 \hline
 $F+\lambda_FK^{-1}$ & $U^{\leq0}$ & $B_{\lambda,\lambda'}$ & $\subseteq U^{\leq0}$ & $B_{\lambda,\lambda'}$  \\
 \hline
 $EK^{-1}+c_FF+c_KK^{-1}$& $U$ & $B_{\lambda,\lambda'}$  & $B_{\lambda,\lambda'}$  & $\subseteq B_{\lambda,\lambda'}$ \\ 
\hline
 \end{tabular}
\end{center}
\end{narrow}~\\

for suitable values $\lambda,\lambda'$ which not necessarily fulfill the condition on $\lambda\lambda'$. Most entries in the table follow simply by subtracting suitable multiples of one another (and ignoring possible $K$-powers). The entry involving $EK^{-1}+c_FF+c_KK^{-1}$ and $K$ comes from the fact that $K$ commutes differently with $E,F$, so suitable commutators return the individual summands, again up to $K$-powers.

Clearly, $U$ is not basic.
When the elements generate some $B_{\lambda,\lambda'}$ but fail the condition on
$\lambda\lambda'$, then as we saw their commutator
$[EK^{-1}+\lambda K^{-1}, F+\lambda'K^{-1}]_{q^2}$ is a linear combination
of $1$ and $K^{-2}$ with non-zero coefficients,
so the generated algebra contains $K^{-2}$. This situation
is very similar to $U$, and indeed since every irreducible $U$-module
$L(k,\pm)$ restricts to an irreducible over this algebra, thus we obtain
irreducible higher-dimensional representations. 

So the only remaining cases are when these two generators are equal, or when
they generate $U^{\geq0},U^{\leq0}$ or when they generate
$B_{\lambda,\lambda'}$.
\end{proof}

In upcoming proofs, we will frequently use restrictions of suitable $U$-modules to construct higher-dimensional representations of some $B$ in question. The restrictions are in general neither irreducible nor semisimple. For a Borel subalgebra, all composition factors need to be one-dimensional.
\begin{example}\label{exm_restriction}
As an explicit example, let $L(\mu),\mu=\alpha/2$ be the 2-dimensional module for $U_q(\sl_2)$ with basis $x_0:=v_{\mu},~x_1:=v_{\mu-\alpha}$. Then the elements of the Borel subalgebra~$B_{\lambda,\lambda'}$ act as follows:
\begin{align*}
(EK^{-1}+\lambda K^{-1}).x_0&=\lambda q^{-1}x_0\hspace{1.8cm} (EK^{-1}+\lambda K^{-1}).x_1=\lambda qx_1+qx_0\\
(F+\lambda'K^{-1}).x_0&=\lambda'q^{-1}x_0+x_1\quad \hspace{1.02cm} (F+\lambda'K^{-1}).x_1=\lambda'qx_1
\end{align*}
Hence $L(\mu)$ restricted to $B_{\lambda,\lambda'}$ has a one-dimensional submodule $\langle x_0+\lambda(1-q^{-2})x_1\rangle$, where the elements $EK^{-1}+\lambda K^{-1},F+\lambda'K^{-1}$ act with eigenvalues $\lambda q,\lambda' q^{-1}$, and a one-dimensional quotient $L(\alpha/2)/(x_0+\lambda(1-q^{-2})x_1)$ with eigenvalues $\lambda q^{-1},\lambda' q$. 
\end{example}

\subsection{Homogeneous Borel subalgebras}\label{sec_standardBorel}

We want to prove that all homogeneous Borel subalgebras of $U=U_q(\g)$ are
standard Borel subalgebras i.e. images of  $U^{\leq 0}$ under a Lusztig automorphism $T_w,w\in W$. 
We start by proving that they are basic:
\begin{lemma}\label{nilpotentede}
 For any Weyl group element $w\in W$ consider the right coideal subalgebra
 $\Uminus{w}$ and the restriction of the counit $\varepsilon:U\to \acf $. Then any
 finite-dimensional irreducible representation $V$ of $\Uminus{w}$ on which all elements in
 $\Uminus{w}\cap \ker(\varepsilon)$ act nilpotently is one-dimensional and hence the
 trivial representation $\acf _\varepsilon$.
\end{lemma}

\begin{proof}
From  \cite[Chp. 6]{HS09} we know, that $U^-[w]$ for $w=s_i v$ with $\ell(w)=1+\ell(v)$ is isomorphic to a smash
product $k[{F}_{\alpha_i}]\#T_{s_i}(U^-[v])$. By applying the anti-algebra isomorphism $\tau:U^+\to U^+$ in \cite[Sec. 4.6]{Jan96} with $T_{s_i}^{-1}=\tau T_{s_i} \tau^{-1}$, we also have that  $\Uminus{w}$ is isomorphic to a smash product  $k[{F}_{\alpha_i}]\#T_{s_i}^{-1}(\Uminus{v})$.\\

We prove our claim by induction on the length $\ell(w)$:
For $w=1$ with $\Uminus{1}=\acf1$ the claim is certainly true. Let the assertion be proven for $v\in W$ and consider the case $w:=s_i v$ with $\ell(w)=1+\ell(v)$. Assume a representation $V$ on which 
$\Uminus{w}\cap \ker(\varepsilon)$ acts nilpotently. Restricting to the subalgebra $T_{s_i}^{-1}(\Uminus{v})$ induction provides that all composition factors are one-dimensional $\acf _\varepsilon$. In particular, there is a common trivial eigenvector of this subalgebra $h.v=\varepsilon(h)v$. Let $V_0\subset V$ 
be the nontrivial subspace 
$$V_0=\{x\in V\mid h.x=\varepsilon(h)x \quad \forall h\in T_{s_i}^{-1}(\Uminus{v})\}$$
This is clearly a $T_{s_i}^{-1}(\Uminus{v})$-subrepresentation. Moreover we know that $\mathrm{ad}_{F_\alpha}$ acts on $T_{s_i}^{-1}(\Uminus{v})$ in the smash product, so 
for any $X_{\mu}\in T_{s_i}^{-1}(\Uminus{v})_{\mu}$
$$X_{\mu}F_{\alpha_i}.x=q^{(\mu,\alpha_i)}F_{\alpha_i}X_{\mu}.x+Y.x\text{ with some }Y\in T_{s_i}^{-1}(\Uminus{v})$$
It follows that $F_{\alpha_i}$ preserves $V_0$ and because $V$ was assumed irreducible
we get $V=V_0$. As $F_{\alpha_i}$ acts on $V$ nilpotently by assumption, it has an eigenvector $x'$ with eigenvalue $0$, which is then a one-dimensional trivial representation of $\Uminus{w}$.
\end{proof}

In particular, we know that all elements $F_{\alpha_i}\in U$ act
nilpotently on a finite-dimensional representation $V$ of $U$, see \cite[Prop. 5.1]{Jan96}. So the previous lemma
shows that all irreducible subquotients of $U$-modules are trivial
one-dimensional. For an arbitrary representation of $\Uminus{w}$ this is certainly
not true, as already the example $\Uminus{s_i}=\acf[F_{\alpha_i}]$ shows, but it
is to expect that $\Uminus{w}$ is still basic. For us, it is however sufficient to show:

\begin{lemma}\label{lm_homogeneousede}
 Given a Weyl group element $w\in W$, $L$ a subgroup of $\Lambda$ and the corresponding right coideal subalgebra $C=\acf [L]\Uminus{w}$ with the property, 
 that for all $\mu\in \Phi^+(w)$ there exists a $\nu\in L$ 
 such that $(\mu,\nu)\neq 0$, then $C$ is basic.
 
 In particular, $U^0\Uminus{w}$ is basic, as well as the right coideal subalgebra
 $U^0\Uplus{w}$.
\end{lemma}
 \begin{proof}
The proof is similar to the proof of Lemma \ref{nilpotentede}: 
We show, that each finite-dimensional irreducible representation
is one-dimensional and that $\Uminus{w}$ acts on all one-dimensional representations by $\varepsilon$.\\

As $\acf [L]$ is abelian, the claim holds for $w=1$. Again we consider inductively the restriction of a $\acf [L]\Uminus{w}$-module to the subalgebra $T_{s_i}^{-1}(\Uminus{v})$ where, by induction, we find a common eigenvector $x$ with $E_\beta.v=0$ and $K_\beta.x=\chi(K_\beta)x$ for all $\beta\in\Phi^+(w)$ and for some character $\chi:L\to\acf ^\times$ extended trivially to $\acf [L]T_{s_i}^{-1}(\Uminus{v})$. Again we consider the nontrivial subspace 
$$V_0=\{x\in V\mid h.x=\chi(h)x \quad \forall h\in \acf[L]T_{s_i}^{-1}(\Uminus{v})\}$$
and again the $\mathrm{ad}_{F_{\alpha_i}}$-stability of $T_{s_i}^{-1}(\Uminus{v})$ in the smash product, and the assumption that $T_{s_i}^{-1}(\Uminus{v})$ acts trivially implies that $V_0$ is a $\Uminus{w}$-submodule. \\

Let $x'$ be an eigenvector of $F_{\alpha_i}$, then the new issue in this proof is whether the eigenvalue $\lambda$ may be nontrivial. But using the assumed element $K_\nu$ with $(\alpha_i,\nu)\neq 0$ with commutator relation $K_\nu F_{\alpha_i}= q^{-(\alpha_i,\nu)}F_{\alpha_i} K_\nu$ with $q$ not a root of unity, implies then an infinite family of eigenvectors with distinct eigenvalues $q^{-n(\alpha_i,\nu)}$, which is a contradiction. Hence $F_{\alpha_i}$ acts on $V_0$ on the eigenvector by $0$ and we have again found a one-dimensional subrepresentation with trivial action of $\Uminus{w}$ as asserted.
\end{proof}

\begin{theorem}
 The basic right coideal subalgebra $U^0U^-$ is maximal with this property
 and hence a Borel subalgebra. We call it the \emph{standard Borel subalgebra}.

Conversely, each homogeneous Borel subalgebra $B$ is as algebra
isomorphic to the standard Borel subalgebra via some $T_w$. Explicitly
$B=\Uminus{w}\,U^0\, S(\Uplus{w^{-1}w_0})$.
\end{theorem}

\begin{proof}
 Since every right coideal subalgebra $B\supset \Uminus{w}U^0$ is by definition
	homogeneous it is sufficient for both assertions to prove that any homogeneous
	basic right coideal subalgebra is contained in some $B=T_w(U^-U^0)$.
 
 By  Theorem \ref{HK11a} every homogeneous right coideal
	subalgebras of
	$U_q^\pm$ is of the form:
 $$B=S(U^-[w_-])\,U^0\,U^+[w_+]
 =\Uminus{w_-}\,U^0\,S(\Uplus{w_+})
 $$
 We are finished if we can prove from $B$ being basic that 
 $\Phi^+(w_+)\cap\Phi^+(w_-)=\{\}$. So assume to the contrary that there is $\mu
 \in \Phi^+(w_+)\cap\Phi^+(w_-)$, so we have root vectors $E_\mu K_\mu^{-1},F_\mu\in
 B$ and hence $K_\mu^{-2}\in B$. Let $\lambda$ be a dominant integral weight
 with $(\lambda,\mu)\neq 0$, then there exists in the $U$-module
 $L(\lambda)$ a $K_\mu^{-2}$-eigenvector with eigenvalue $\neq 1$, say the
 highest weight vector. But $B$ contains a copy of $U_q(\sl_2)$ and all
 one-dimensional representations of this requires $K_\mu^{-2}$ to act with eigenvalue~$=1$.
 Thus the restriction of $L(\lambda)$ to $B$ has to have some
 higher-dimensional irreducible composition factor (note that the eigenvalue
 argument does not require semisimplicity). Hence $B$ is in this case not
 basic. 
\end{proof}

\section{Induction of one-dimensional characters}\label{sec_induction}

\subsection{Definition and first properties}

One of the reasons Borel subalgebras of a Lie algebra or quantum group are
interesting is because they can be used to construct induced modules and a
category $\mathcal{O}$. An interesting implication of defining and classifying
unfamiliar Borel subalgebras $B$ is to study the respective induced
representations from $B$ to $U=U_q(\g)$ of any one-dimensional $B$-module
$\acf _\chi$:
$$V(B,\chi):=U \otimes_B \acf _\chi$$

As defining property for Borel subalgebras $B$ we chose to generalize the Lie
algebra term ``maximal solvable'' by ``maximal with the property \emph{basic}''
i.e. all irreducible finite-dimensional representations are one-dimensional.
This matches the upcoming purpose since we can again use solely
one-dimensional characters of $B$. For a right
coideal subalgebra $B$ the category $\Rep(B)$ is not a tensor category, but a right $\Rep(H)$-module category, where $V\in\Rep(H)$ acting on $M\in\Rep(B)$ gives $M\otimes_\acf V$ with $B$-action given by $b^{(1)}\otimes b^{(2)}\in C\otimes H$.

\begin{lemma}
 For a right coideal subalgebra of a Hopf algebra $B\subset H$,
 the induction functor ${\Rep(B)\to \Rep(H)}$ and the restriction functor $\Rep(H)\to \Rep(B)$ are both morphisms of right $\Rep(H)$-module categories.
\end{lemma}
\begin{proof}
 Restriction is a morphism of module categories with trivial structure map
	${\Res(V\otimes W)=\Res(V)\otimes W}$ due to the right coideal subalgebra
	property. Similarly, Induction is a morphism of module categories with structure map given by
 $$(H\otimes_B M)\otimes V\cong H\otimes_B(M\otimes V)$$
 $$h\otimes m\otimes v\mapsto h^{(1)} \otimes m\otimes S(h^{(2)}).v$$
 which is easily checked to be an $H$-module morphism. 
\end{proof}
\begin{remark}
 In our examples, $H$ is free as a $B$ module, which is very helpful in constructing the induced representations. This should be an instance of the celebrated result by Skryabin \cite{Skry06} that a Hopf algebra is free over its coideal subalgebras, which assumes finite dimension. It would be helpful to have a version that applies in our case. We think that $B$ being free over $U^0$ is the decisive property. 
\end{remark}

Our main interest is how induced representations of Borel subalgebras
decompose as $U$-modules. For a first result, we use a part of our defining property: 

\begin{lemma}\label{lm_InductionQuotientExistence}
 Let $H$ be a Hopf algebra and $B\subset H$ a basic right coideal subalgebra.
 Then any finite-dimensional irreducible $H$-module $V$
 is a quotient of an induced module
 $H \otimes_B \acf_\chi$ for some character $\chi$.
\end{lemma}

\begin{proof}
 Consider the restriction of the finite-dimensional module $V$ to $B$. Since $B$
 is basic, the composition series of $V$ consists of one-dimensional
 $B$-modules, and in particular there is a nontrivial $B$-module monomorphism
 $\acf_\chi\to V$ for a suitable $\chi$. This induces up to a nonzero $H$-module
 morphism $H \otimes_B \acf_\chi\to V$. For $V$ irreducible this morphism has to be surjective.
\end{proof}

\begin{example}
 Let $B=U^{\leq 0}$ be the standard Borel subalgebra of $U$, then
 every one-dimensional character $\chi:B\to \acf$ is zero on $U^-$ and thus
 comes from some group character $\chi:\Lambda\to \acf^\times$. The induced module
 $U\otimes_B \acf_\chi$ is isomorphic to $U^+$ as a
 $U$-module. It is a lowest-weight module generated by the vector $v$
 with $K_\mu.v=\chi(K_\mu)v$ and $F_\alpha.v=0$ for all $\alpha\in \Pi$.
\end{example}

We shall see that also for our new Borel subalgebras $B$ the induced modules for generic $\chi$ are irreducible, while special values of $\chi$ have as quotient each of the irreducible finite-dimensional $U$-modules $L(\lambda)$. However, the action of $U^0$ will not be diagonalizable anymore.\\

\begin{problem}
 Is it true that any induced module from a Borel subalgebra has a unique irreducible quotient? It is to expect that maximality enters here.
\end{problem}
\begin{problem}
 Can the graded algebra in the upcoming Conjecture A be used to understand these modules and determine the decomposition behavior of their induced modules? We would expect that $\gr(B)^0$ plays a similar role as $U^0$ in the standard case, so the graded modules $\gr(V)$ should be $\gr(B)^0$-diagonal.
\end{problem}
\begin{problem}
 Can a category $\mathcal{O}$ with nice properties be defined? In particular, it should be an abelian category, closed under submodules and quotients, with enough projectives and injectives. 
\end{problem}
%
%
%

\subsection{Example \texorpdfstring{$A_1$}{A1}}

We now want to construct and decompose all induced representations in the case $\sl_2$. From Section \ref{sec_exampleA1} we know all Borel subalgebras of $U_q(\sl_2)$ up to reflection: \\

In the familiar case $U^{\leq0}=\langle K,F \rangle$ all one-dimensional
representations are of the form $\chi(F)=0$ and $\chi(K)$ arbitrary. Then the
induced module is
$$V(U^{\leq0},\chi):=U \otimes_{U^{\leq0}} \acf_\chi\cong \acf[E]1_\chi$$
with lowest weight vector $K.1_\chi=\chi(K)1_\chi$. All of these modules have a diagonal action of $K$. For certain integral choices $V(U^{\leq0},\chi)$ has a finite-dimensional irreducible quotient and all finite-dimensional irreducible $U$-representations arise this way.\\ 

The novel case in Example \ref{Weyl algebra} is $B_{\lambda,\lambda'}:=\langle E^\phi,F^\phi\rangle$ with $E^\phi:=EK^{-1}+\lambda K^{-1},\, F^\phi=F+\lambda'K^{-1}$ and arbitrary scalars $\lambda\lambda'=\frac{q^2}{(1-q^2)(q-q^{-1})}$, an algebra isomorphic to the quantized Weyl algebra. As we showed, all finite-dimensional irreducible representations factorize over the commutative quotient 
$$\acf[E^\phi,F^\phi ] /( E^\phi F^\phi-\frac{q^2}{(q-q^{-1})(1-q^2)})$$
Hence all 
 one-dimensional representations are of the form $\chi(E^\phi)=e$, $\chi(F^\phi)=f$ with $ef=\frac{q^2}{(q-q^{-1})(1-q^2)}=\lambda\lambda'$.
Using the PBW basis we easily get 
$$V(B_{\lambda,\lambda'},\chi):=U \otimes_{B_{\lambda,\lambda'}} \acf_\chi\cong \acf[K,K^{-1}]1_\chi$$
and we calculate the action to be explicitly
\begin{align*}
 K.K^n1_\chi
 &=K^{n+1}1_\chi\\
 &\\
 F.K^n1_\chi
 &=q^{2n}K^nF1_\chi\\
 &=q^{2n}f\cdot K^n1_\chi-q^{2n}\lambda' \cdot K^{n-1}1_\chi\\
 &\\
 E.K^n1_\chi
 &=q^{-2n-2} K^{n+1}EK^{-1}1_\chi\\
 &=q^{-2n-2}e\cdot K^{n+1}1_\chi-q^{-2n-2}\lambda \cdot K^{n}1_\chi
 \end{align*}
or expressed as  matrices 
 \begin{align*}
K&=\begin{pmatrix}
 \ddots &&&&\\
 &0&0 &0 &\\
 & 1 & 0 & 0 &\\
 & 0 & 1 & 0 &\\
 &&&& \ddots 
\end{pmatrix}\\
F&=\begin{pmatrix}
 \ddots &&&&\\
 &q^{2(n-1)}f & -q^{2n}\lambda' & 0 &\\
 &0 & q^{2n}f & -q^{2(n+1)}\lambda' \\
 &0&0& q^{2(n+1)}f &\\
 &&&& \ddots 
\end{pmatrix}\\
E&=\begin{pmatrix}
 \ddots &&&&\\
 &-q^{-2(n-1)-2}\lambda&0&0&\\
 & q^{-2(n-1)-2}e & -q^{-2n-2}\lambda &0& \\
 & 0 & q^{-2n-2}e & -q^{-2(n+1)-2}\lambda &\\
 &&&& \ddots 
\end{pmatrix}\\
\end{align*}

In particular, the action of $K$ is not diagonalizable. Moreover the action of
$K$ shows that no $V(B_{\lambda,\lambda'},\chi)$ has finite-dimensional proper
submodules.

We can determine all submodules with a trick that does not seem to easily
generalize beyond rank one:

\begin{lemma}
The induced $U_q(\sl_2)$-module with respect to a Borel subalgebra $B_{\lambda,\lambda'}$
 $$V(B_{\lambda,\lambda'},\chi):=U \otimes_{B_{\lambda,\lambda'}} \acf_\chi\cong \acf[K,K^{-1}]1_\chi$$
has a nontrivial submodule $V'$ iff for some $n\in\N_0$ 
 $$\chi(E^\phi)=\epsilon q^{n}\cdot \lambda, \text{ equivalently } \chi(F^\phi)=\epsilon q^{-n}\cdot \lambda'$$ 
It is cofinite of codimension $[V:V']=n+1$.
\end{lemma}
\begin{proof}
 A $U_q(\sl_2)$-submodule $W\subset V(B_{\lambda,\lambda'},\chi)\cong \acf[K,K^{-1}]$ is in particular a $\acf[K,K^{-1}]$-submodules under left-multiplication i.e. an ideal (instead of weight-spaces). Since this is a principle ideal ring, there exists a Laurent-Polynomial $P(X)=\sum_{n\in \Z}^{finite} c_nX^n$ with $W=(P),X=K$. This is a submodule iff $E.P$ and $F.P$ are multiples of $P$. We calculate explicitly:
 \begin{align*}
 (F.P)(X)
 &=\sum_n c_n q^{2n}f\cdot X^n-q^{2n}\lambda' \cdot X^{n-1}\\
 &=P(q^2X)\left(f-\lambda' X^{-1}\right)\\
 (E.P)(X)
 &=\sum_n c_n q^{-2n-2}e\cdot X^{n+1}-q^{-2n-2}\lambda \cdot X^{n}\\
 &=P(q^{-2}X)q^{-2}\left(eX-\lambda\right)
 \end{align*}
 For degree reasons, this can only give a multiple of $P$ if for some $a,b,c,d$:
 \begin{align*}
 (aX+b)P(X)&\stackrel{!}{=}P(q^2X)\left(fX-\lambda' \right)\\
 (cX+d)P(X)&\stackrel{!}{=}P(q^{-2}X)q^{-2}\left(eX-\lambda\right)
 \end{align*}
 Either $P(X)$ is constant, then $W$ is the entire module, or the zeroes of $P$ have to lay in a chain $q^{-n}X_0,q^{-n+2}X_0,\ldots q^{n}X_0$ for some $n\in \N_0$, with 
 $q^nX_0=\lambda'/f$ and $q^{-n}X_0=\lambda/e$. 
 The quotient module is thus the $(n+1)$-dimensional ring extension $\acf[K,K^{-1}]/(P)$, where $K$ acts again by left-multiplication.
\end{proof}

We show a different and quite general idea to detect the finite-dimensional quotient modules (i.e. the cofinite submodules) and circumvent the problem of non-diagonalizable $K$-action: To find finite-dimensional quotient modules we need to find elements in $\Hom_{U}(U\otimes_B \acf_\chi)$ and this can be done by decomposing 
 $\Hom_{\acf}(U,\acf)$ into irreducible $L(\lambda)$. This is now possible since $\Hom_{\acf}(U^0,\acf)$ is diagonalizable under left-multiplication ---  morally because it is the closure. The following Lemma verifies explicitly the existence assertion in Lemma \ref{lm_InductionQuotientExistence}. 

\begin{lemma}\label{lm_quotientIndA1}
 The induced $U_q(\sl_2)$-module with respect to a Borel subalgebra $B_{\lambda,\lambda'}$
 $$V(B_{\lambda,\lambda'},\chi):=U \otimes_{B_{\lambda,\lambda'}} \acf_\chi\cong \acf[K,K^{-1}]1_\chi$$
 has a finite-dimensional quotient, namely the irreducible module $L(\epsilon,n)$ of dimension $n+1$ and sign $\epsilon=\pm 1$ for precisely one choice of $\chi$, namely 
 $$\chi(E^\phi K^{-1})=\epsilon q^{n}\cdot \lambda, \text{ equivalently } \chi(F^\phi)=\epsilon q^{-n}\cdot \lambda'$$ 
\end{lemma}
\begin{proof}
 Let $g:V(B_{\lambda,\lambda'},\chi)\to L(\epsilon, n)$ be a (nonzero) module homomorphism to the irreducible highest weight module with highest weight vector $v_0$ for highest weight $Kv=q^nv$.
 It has a basis $v_0,\ldots v_n$ with $Kv_i=\epsilon q^{n-2i}v_i$. 
 Denote the coefficients of $g$ in this basis by $g_k,\;k=0\ldots n$,. Then by definition
 $$\epsilon q^{n-2k}g_k(K^i 1_\chi)=K.g_k(K^i1_\chi)=g_k(K^{i+1}1_\chi)$$
 Thus $g$ is fixed by its image $g_k(1_\chi)$ via
 $$g_k(K^i1_\chi)=\epsilon^i q^{(n-2k)i}g_k(1_\chi)$$
 and since the map should be surjective we need all $g_k(1_\chi)\neq 0$ for $k=0\ldots n$.\\
 
 On the other hand, the action of $F$ and $E$ demands:
 \begin{align*}
 g{k-1}(K^i1_\chi)
 &=g_k(F.K^i1_\chi)\\
 &=q^{2i}fg_k(K^i1_\chi)-q^{2i}\lambda'\chi_k(K^{i-1}1_\chi)\\
 &=q^{2i}\left(\epsilon^{i}q^{(n-2k)i}f-\epsilon^{i-1}q^{(n-2k)(i-1)}\lambda'\right)g_{k}(1_\chi)\\
 \epsilon[k+1]_q[n-k]_q g_{k+1}(K^i1_\chi)
 &=g_k(E.K^i1_\chi)\\
 &=q^{-2i-2}e g_k(K^{i+1}1_\chi)-q^{-2i-2}\lambda g_k(K^{i}1_\chi)\\
 &=q^{-2i-2}\left(\epsilon^{i+1}q^{(n-2k)(i+1)}e-\epsilon^i q^{(n-2k)i}\lambda'\right)g_{k}(1_\chi) 
 \end{align*}
 and the boundary conditions of the first resp. second equation for $k=0$ resp. $k=n$ that the right-hand side has to be zero
 \begin{align*}
 0&= \epsilon f-q^{-n}\lambda'\\
 0&= \epsilon q^{-n}e-\lambda
 \end{align*}
 which is the asserted condition. One could also derive direct formulae for $g_k(K^i 1_\chi)$ and in this case it is easy to see that $g$ above is indeed a module homomorphism.
 
\end{proof}

\section{The structure of the graded algebra of a right
coideal subalgebra}

\subsection{Conjecture A}

By Theorem~\ref{HK11b_shifted}, the right coideal subalgebras $C$ of $U^{\leq 0}$ with $C\cap U^0$ a Hopf algebra are
character shifts $\Uminus{w}_{\phi}\acf[L]$ with $L\subset \Lambda$ orthogonal
to $\supp(\phi)$. These character shifts are not $\Lambda$-graded in general. To further understand these algebras and their representation theory, we propose to study their associated graded algebras, which are again right coideal subalgebras. As a first step, we present and prove in many cases the following conjecture, that states morally that up to localization (which enforces that again  $\gr(C)\cap U^0$ is a Hopf algebra), the graded algebra is given by $\gr(\Uminus{w}_\phi)=\Uminus{w'}$, 
with an explicit formula of $w'$. We now give a precise statement: \\

Fix some $w\in W$ and some character $\phi$ of $\Uminus{w}$ and consider the right coideal subalgebra $\Uminus{w}_{\phi}$. We call  the following the \emph{standard $\Z$-grading} of $U$:
$$\deg(E_{\alpha_i})=1, \quad \deg(K_{\alpha_i})=0,\quad \deg(F_{\alpha_i})=-1$$
Consider the graded algebra $\gr(\Uminus{w}_{\phi})$ with respect to the standard $\Z$-grading. For any nonzero coset in the graded algebra,  the term in leading degree is independent of the choice of a coset representative. This gives rise to a canonical algebra and comodule map
$$f:\gr(\Uminus{w}_{\phi})\to U^{\leq0}$$ 
In particular, the image $D$ of $f$ is a $\mathbb{Z}$-graded right coideal subalgebra of $U^{\leq0}$. Since subalgebras give rise to subalgebras in the graded algebra, $D^0=D\cap U^0$ is the span of an abelian semigroup
$$G(D^0)=\langle K_{\mu}^{-1} \mid \mu \in \supp(\phi)\rangle$$
We denote by $\tilde{G}$ its quotient group or localization. As localization of $D$, we consider $\acf[\tilde{G}]D$, which is by Theorem \ref{HK11b_nonshifted} of the form $\acf[\tilde{G}]\Uminus{w'}$ for some Weyl group element $w'$ (see Proposition \ref{prop_localization}), for which we now conjecture an explicit formula:\\

\begin{conjectureA}\label{conj_A}
For any $w\in W$ and character $\phi:\Uminus{w}\to \acf$, the localization of the graded algebra $D$ of $\Uminus{w}_{\phi}$ is given by
 $$\acf[\tilde{G}]D=\acf[\tilde{G}]\Uminus{w'}.$$
for the explicit Weyl group element of length $\ell(w)-|\supp(\phi)|$
$$w'=(\prod_{\beta\in \supp(\phi)}s_\beta)w$$
\end{conjectureA}
Note that the roots $\beta\in\supp(\phi)$ in Theorem \ref{HK11b_shifted} are pairwise orthogonal, so the elements $w_\beta$ commute and the definition of $w'$ does not depend on the order of the product. Note that the element $w'$ appears in \cite[(3.6)]{HK11b}.\\

%

 We first give some examples and formulate the consequence for triangular coideal subalgebras.

\begin{example}
 Clearly, for $\phi=\varepsilon$ the right
 coideal subalgebra $\Uminus{w}_{\phi}$ is $\Lambda$-graded and in particular $\Z$-graded, and we have $w'=w$.
\end{example}
\begin{example}\label{exm_gradedsl2}
Let $\g=\sl_2$, consider $\Uminus{s_1}=\langle F\rangle$, and consider for the character $\phi(F)=\lambda'$ the character shift 
$$\Uminus{s_1}_\phi=\langle F+\lambda'K^{-1}\rangle$$
Then for $\lambda'\neq 0$ the leading term is $\lambda'K^{-1}$ an thus the graded algebra is as conjectured with $w=s_{1},w'=1$
 \begin{align*}
 \gr(\Uminus{s_1}_\phi)\cong \Uminus{1}\acf[K^{-1}]&=\langle K^{-1} \rangle
 \end{align*}
\end{example}
\begin{example}\label{exm_A2ConjectureA}
 Let $\g=\sl_3,w=w_0=s_1s_2s_1,\,w'=s_1w=s_2s_1$ and
 $\phi(F_1)=\lambda'$, $\supp(\chi)=\{\alpha_1\}$. Then $\Uminus{w}_\phi=\langle
 F_1+\lambda'K_1^{-1},F_2 \rangle$. The leading degree terms of the character shift
 $F^\phi_1=F_1+\lambda'K_1^{-1}$ is $\lambda'K_1^{-1}$, and
 $F^\phi_2=F_2$. Typically, character shifts of elements of $\Lambda$-degree $-(\alpha_1+\alpha_2)$ have leading degree terms
 in degree $-\alpha_2$:
 \begin{align*}
 {F_1F_2}^\phi=F_1F_2+\lambda' K_1^{-1}\cdot F_2\\
 {F_2F_1}^\phi=F_2F_1+F_2\cdot \lambda' K_1^{-1}
 \end{align*}
 However, in a suitable linear combination of these two elements, the leading degree terms of the
 character shifts cancel and we again land in $\Lambda$-degree $-(\alpha_1+\alpha_2)$:
 $$(F_1F_2-q^{-1}F_2F_1)^\phi=F_1F_2-q^{-1}F_2F_1$$
 Thus the graded algebra is as conjectured, with $w'=s_1w=s_1s_2$: 
 $$f:\;\gr\;\Uminus{s_1s_2s_1}_\phi \cong \Uminus{s_2s_1}\acf[K_1^{-1}]$$
\end{example}

In contrast, for a right coideal subalgebra in the positive part $U^0S(U^+)$, the
graded right coideal subalgebra has the same degrees, i.e. by choice of our grading the leading terms of the character shifted $E_\alpha^\phi$ are simply the $E_\alpha$ themselves.
$$f:\gr(\Uplus{w}_{\phi})\stackrel{\sim}{\longrightarrow} \Uplus{w}$$


The two statements can be easily combined:

\begin{corollary}\label{cor_conjA}
Assume that Conjecture A holds for $\Uminus{w_-}_{\phi_-}$ i.e. that for the localization of $\gr(\Uminus{w_-}_{\phi_-})$ holds:
\begin{align*}
\acf[\tilde{G}]\;\gr(\Uminus{w_-}_{\phi_-})&\cong\acf[\tilde{G}]\;\Uminus{w'} 
\end{align*}
Then for any triangular right coideal subalgebra $C$ with $C\cap U^0$ a Hopf algebra, the map sending elements to their leading degree terms gives a description of the localization of $\gr(C)$ as follows: 
\begin{align*}
\acf[\tilde{G}]\;\gr\left(\Uminus{w_-}_{\phi_-}\;\acf[L]\;S(\Uplus{w_+})_{\phi_+}\right)
&\cong \Uminus{w'_-}\;\acf[\tilde{G}L]\;S(\Uplus{w_+}) 
\end{align*}
\end{corollary}
\begin{example}
Consider our example $B_{\lambda,\lambda'}$, which has a triangular decomposition as in Corollary \ref{cor_triangular} for parameters $w_-=w_+=s_1$
$$B_{\lambda,\lambda'}=\Uminus{s_1}_{\phi_-}\acf[1]S(\Uplus{s_1})_{\phi_+}=\langle F+\lambda'K^{-1},EK^{-1}+\lambda F^ {-1}\rangle$$
The leading term in the negative part is $\lambda'K^{-1}$ according to Conjecture A, as discussed in Example \ref{exm_gradedsl2} above. The leading term in the positive part is as always the unshifted element $EK^{-1}$.   Altogether, the graded algebra of $B_{\lambda,\lambda'}$ is
 \begin{align*}
 \gr(B_{\lambda,\lambda'})=\Uminus{1}\acf[K^{-1}]S(\Uplus{s_1})&=\langle K^{-1},EK^{-1} \rangle
 \end{align*} 
which is up to localization, i.e. adding the generator $K$, the upper Borel part $U^{\geq 0}$.
\end{example}
This description leads us to the question and a conjecture for which choice of $(w_+,\phi_+,w_-,\phi_-)$ the construction produces a \emph{basic} right coideal subalgebra. We formulate this in the next section as Conjecture B and we will prove one direction, namely a criterion which choice leads to a non-basic right basic coideal subalgebra. Let us briefly sketch how we arrive at the conjecture and what is the big picture we would ultimately hope for. First, we find it natural to expect the basicness of the right coideal subalgebra to be closely related to the basicness of the graded localized right coideal subalgebra $\Uminus{w_-'}\;\acf[\tilde{G}L]\;S(\Uplus{w_+})$. Second, we find it natural to expect that such graded coideal subalgebras are basic iff $\ell(w'_-)+\ell(w_+)=\ell(w'_-w_+)$, which in the classical Lie algebra picture should mean that the semisimple part of the Levi decomposition is trivial. We would be thankful for remarks or references on general statements in either of these questions.

\subsection{Localization}

Let in this section $A$ be a $\Z$-filtered right coideal subalgebra of $U$. A technical difficulty, which we address in this section is that $A'=\gr(A)$ has in general $A'^0=A'\cap U^0=\acf[G]$ where  $G\subset \Z^k$ is just a semigroup. Such a $\Z$-graded right coideal subalgebra $A'$ could be not even $\Z^k$-graded. Differently spoken, it is not necessarily a free module over the semigroup $G$. 

\begin{example}\label{exm_nongraded}
 Consider the subalgebra in $U_q(\sl_3)^-$ generated as: 
 $$C = \langle F_1 K_3^{-1} + F_3 K_1^{-1},K_3^{-1},K_1^{-1}\rangle$$
 Computing the coproduct of each generator quickly shows that these elements generate a right coideal subalgebra. To compute the generated algebra, note that multiplication with $K_1^{-1}$ or $K_3^{-1}$ from either side commutes differently with $F_1,F_3$, so the following elements are also contained in $C$ as suitable linear combinations 
 $$C \ni F_1K_3^{-1}K_1^{-1},\,F_3K_1^{-1}K_3^{-1}$$
 However, the individual summands $F_1K_3^{-1}$ and $F_3K_1^{-1}$ of the initial generator are not contained in $C$. 
\end{example}

These $\Z$-graded right coideal subalgebras $A'$ are not covered by the classification Theorem \ref{HS09}. 
We shall not try to classify $A'$ directly, but study its localization:
\begin{proposition}\label{prop_localization}
 Let $A'\subset U^{\leq 0}$ be a right coideal subalgebra that is graded with respect to the standard $\Z$-grading in this section. Define the semigroup $G$ by $A'^0=A'\cap U^0=\acf[G]$, then $\acf[G]A'=A'\acf[G]$, and thus we can define a localization in the noncommutative setting \cite[Prop. 1.4]{Sten75}: 
 
 Let $\tilde{G}\subset \Z^k$ be the quotient group (localization) of $G$. Then $\tilde{A'}:=\acf[\tilde{G}]A'=A'\acf[\tilde{G}]$ is again a right coideal subalgebra and there exists some $w'\in W$ with 
 $$\tilde{A'}=\acf[\tilde{G}]\Uminus{w'}$$ 
\end{proposition}
\begin{proof}
 Since $A'$ is an algebra we have $KA'=A'K$ for $K\in G$, 
 which implies the first claim.
 Now, $\tilde{A'}$ is $\Z$-graded and $\tilde{A'}^0=\acf[\tilde{G}]$ is a group ring. So it fulfills the conditions of Theorem \ref{HK11b_shifted} and is thus a character shift. But since they are $\Z$-graded, the character shift has to be trivial, and hence $\tilde{A'}$ is as asserted, in particular $\Z^n$-graded.
\end{proof}

\begin{proposition}\label{prop_highpower}
 Let $G\subset \Z^k$ be a semigroup (or any other abelian cancellative semigroup) and let $\tilde{G}$ be its group of fractions. Let $\tilde{V}$ be a $\tilde{G}$-module and $V$ be a $G$-submodule, such that acting with $\tilde{G}$ generates the entire module $\tilde{G}V=\tilde{V}$. Then for any $v\in \tilde{V}$ there exists a $K\in G$ such that $Kv\in V$.
\end{proposition}
\begin{proof}
 Since we assumed $\tilde{G}V=\tilde{V}$, there exist finitely many elements $K^{(i)}\in \tilde{G}$ and vectors $v^{(i)}\in V$ such that 
 $$\sum_{i=1}^n K^{(i)} v^{(i)}=v$$
 Now by definition of $\tilde{G}$ there exists for each element $K^{(i)}$ an element $K_{(i)}\in G$ with $K_{(i)}K^{(i)}\in G$. Taking some $K\in G$ in $\bigcap_i K_{(i)}G$ (for example their product) proves the assertion. 
\end{proof}
Applying this to $\acf[\tilde{G}]A'=\acf[\tilde{G}]\Uminus{w'}$ we get the following easy consequence, which will be useful later on, and which is illustrated again by Example \ref{exm_nongraded}:
\begin{corollary}
 For any root $\mu\in \Phi(w')$ and a reduced expression for $w'$, there exists a $K\in G$, such that $KF_\mu\in A'$. 
\end{corollary}

\subsection{Growth conditions}

Recall that we are in the situation that for every ${w\in W}$ and every character $\phi$ we have, after localization, by Proposition \ref{prop_localization},
$$\acf[\tilde G]\,\gr(\Uminus{w}_\phi)=\acf[\tilde{G}]\,\Uminus{w'}$$
for some $w'\in W$ and semigroup $G\subset \Z^n$, and the content of Conjecture A is an explicit formula for $w'$.\\

 The goal of this section is to compute the length of the unknown Weyl group element $w'$ by comparing the growth of both sides of the isomorphism above resp. the number of PBW-generators. It would be desirable to argue here with a well-established concept such as Gelfand-Kirillov dimension. The Gelfand-Kirillov dimension is invariant under the operation of passing to the graded algebra \cite{MS89} (under suitable conditions, which hold in our setting) and invariant under the operation of passing to the noncommutative localization \cite{LMO88} (the conditions in this source however do not quite hold in our infinite-dimensional setting). The result of this subsection would similarly follow from such a line of argument. 
Since however our reasoning does not depend on the algebra structure but is solely a growth condition on the filtered vector space, we decide to not pursue the Gelfand-Kirillov argument and rather argue in a more ad-hoc fashion:

\begin{definition}
 Let $V$ be an $\N$-filtered vector space with $V_n/V_{n-1}$ finite-dimensional. Let $\mathcal{H}(V,z)=\sum_{n\geq 0} \dim(V_n/V_{n-1})z^n$ be the Hilbert series, then we define  $\mathrm{growth}(V)$ to be the value $a$ such that 
 $$0<\mathrm{lim\;inf}_{z\to 1} (1-z)^a\mathcal{H}(V,z)
 \quad\text{and}\quad \mathrm{lim\;sup}_{z\to 1} (1-z)^a\mathcal{H}(V,z) < +\infty$$
 If the series does not converge to an analytic function in a vicinity of $z=1$ or the above condition is not met for any $a$, we say the growth is undefined. 
\end{definition}
The sole purpose of this definition is that $V^{sub}\subset V \subset V^{sup}$ with well-defined value $\mathrm{growth}(V^{sub})=\mathrm{growth}(V^{sup})=:a$ implies an intermediate well-defined $\mathrm{growth}(V)=a$. Note that $\mathcal{H}(V,z)$ needs not be a rational function, even if $\mathcal{H}(V^{sub},z),\mathcal{H}(V^{sup},z)$ are; note further that the values of the limits need not coincide. 

\begin{example}
 If $A'$ is a graded algebra with a basis of sorted monomials in generators $x_1,\ldots x_\ell$ of degree $d_1,\ldots d_\ell \geq 1$, then $\mathrm{growth}(V)=\ell$, since explicitly
 $$\mathcal{H}(A',z)=\prod_{i=1}^\ell \frac{1}{1-z^{d_i}}$$
\end{example}
We now apply this notion and result to our $A=\Uminus{w}_\phi$ and $A'=\gr(A)$ in $U^-$. Technically, we need to consider a \emph{modified grading}, which has finite-dimensional homogeneous components, but equal behavior with respect to character shifts: Define the modified degree $\mathrm{mdeg}(-)=\deg(-)+\mathrm{sdeg}(-)$ by adding a second degree $\mathrm{sdeg}(F_i)=\mathrm{sdeg}(K_i^{\pm 1})=-1$, so $\mathrm{mdeg}(F_i)=-2$ and $\mathrm{mdeg}(K_i^{\pm 1})=-1$. Note that in this finer grading, the algebra $A'$ need not be graded anymore. Clearly, the homogeneous components of $\mathrm{mdeg}(-)$ on $U^-$ are finite-dimensional. Since characters shifts are graded with respect to $\mathrm{sdeg}(-)$, this new definition does not change the vector space $\gr(A)$:
$$\gr_{\deg(-)}(A)=\gr_{\mathrm{mdeg}(-)}(A)$$
\begin{example}\label{exm_somegrowths}
 Let $A=\Uminus{w}$ with the modified grading, then ${\mathrm{growth}(A)=\ell(w)}$. More generally, let $G$ be a semigroup with $\tilde{G}=\Z^k$ and take $A'=\acf[G]\Uminus{w'}$, then $\mathrm{growth}(A')=\rank{\tilde{G}}+\ell(w')$. 
\end{example}

We now need to compare the growth of $A=\Uminus{w}_\phi$, or equivalently $A'=\gr(A)$, to its localization $\acf[\tilde{G}]A'=\acf[\tilde{G}]\Uminus{w'}$: 


\begin{lemma}
 Let $V$ be a $\Z$-filtered module over $A'$ with the filtration in this section. Then we have for the localization $\mathrm{growth}(V)=\mathrm{growth}(\acf[\tilde{G}]V)$. 
\end{lemma}
\begin{proof}
 Let $\tilde{V}=\tilde{G}V$ as in Proposition \ref{prop_highpower}, which is a free $\acf[\tilde{G}]$-module. Let $I$ be a $\tilde{G}$-basis of $\tilde{V}$, then by Proposition \ref{prop_highpower} there are elements $K_{(i)}\in G$ with $K_{(i)}v_i\in V$ and we define its span to be $V^{sub}\subset V$. On the other hand, we may consider $V^{sup}=\tilde{V} \supset V$.

 Then we have an inequality by definition 
 $$\mathrm{growth(V^{sub})}\leq \mathrm{growth(V)} \leq \mathrm{growth(V^{sup})}$$
 but in our case $\mathrm{growth(V^{sub})}=\mathrm{growth(V^{sup})}$, which proves the assertion
\end{proof}

\begin{corollary}\label{cor_samesize}
 Let now again $A:=\Uminus{w}_{\phi}$ and $A':=\gr(A)$ and the localization $\acf[\tilde{G}]A'=\acf[\tilde{G}]\Uminus{w'}$. Then comparing the growth gives
 $$\ell(w)=\rank(\tilde{G})+\ell(w')$$
\end{corollary}

\subsection{An inductive approach to Conjecture A}

In this section we attempt to prove Conjecture A for some given $w\in W$ under the inductive  assumption that Conjecture A holds for all $u\in W$ with $\ell(u)<\ell(w)$ and all characters. Indeed, we can prove Conjecture A for $w$ to holds except for a very specific situation for $w,w',u,u'$. We will analyze these specific situations in subsequent sections.\\

Let $w\in W$ and $\phi$ be a character on $\Uminus{w}$ with support $\supp(\phi)\subset \Phi^+(w)$, a set of orthogonal simple roots. We continue to study the character shifted coideal subalgebra, its graded algebra, and its localization
$$A:=\Uminus{w}_\phi,\qquad A':=\gr(A),\qquad \tilde{A'}=\acf[\tilde{G}]A'=\acf[\tilde{G}]\Uminus{w'}$$
for some $w'\in W$, whose explicit description is the goal of Conjecture A. We attempt an induction on the length $\ell(w)$. As the induction base $w=e,\,\ell(w)=0$, we have $U^-[e]=\C1$ and trivial character shift $\phi=\varepsilon,\,G=1$, so $w'=e$ holds according to Conjecture A. 
 
\begin{lemma}\label{lm_inductionstep}
 Assume Conjecture A holds for all $u\in W$ with $\ell(u)<\ell(w)$ and all characters. Assume 
 $w=us_i$ with $\ell(w)=\ell(u)+1$ for some $\alpha_i$.
 Then one of the following cases occurs:
 \begin{itemize}
 \item If $u(\alpha_i)\in\supp(\phi)$ then Theorem \ref{conj_A} holds for $w$ and $w'=u'$. Hence Conjecture A holds for $w'$.
 \item If $u(\alpha_i)\not\in\supp(\phi)$ then there exists a simple root $\alpha_j$ such that $w':=u's_j$ has length $\ell(u')+1$ and 
 $$\acf[\tilde{G}(A'^0)]A'=\acf[\tilde{G}(A'^0)]\Uminus{w'}$$
 Conjecture A holds for $w$ iff $\alpha_j=\alpha_i$. 
 \end{itemize}
\end{lemma}

\begin{proof}
For $u<w$ the restriction $\phi_u$ of $\phi$ to $\Uminus{u}\subset \Uminus{w}$ is again a character with $\supp(\phi_u)\subseteq \supp(\phi)$. The inductive assumption that Theorem \ref{conj_A} holds for $\Uminus{u}_{\phi_u}$ means that we have 
\begin{align*}
 A_u' 
 &:=\gr\left(\Uminus{u}_{\phi_u}\right)\\
 A'^0_u
 &=\acf[\langle K_\mu^{-1},\; \mu\in \supp(\phi_u)\rangle]\\
 \acf[\tilde{G}_u]A_u 
 &\cong \acf[\tilde{G}_u] \Uminus{u'} 
\end{align*}
where $u'$ has the asserted form $u':=(\prod_{\beta\in \supp(\phi_u)}s_\beta)u$, and where $\tilde{G}_u$ is the quotient group of the semigroup $G_u$ generated by $\supp(\phi_u)$, with $A'^0_u=\acf[G_u]$.\\

Now for the new element $w\in W$ with $\Phi^+(w)=\Phi^+(u)\cup \{u(\alpha_i)\}$ a new character shifted root vector $F^\phi_{u(\alpha_i)}$ appears in the PBW-basis
\begin{align*}
 A&:=\Uminus{w}_{\phi}
 =\Uminus{u}_{\phi_u}\left\langle {F}_{u(\alpha_i)}^\phi\right\rangle\\
 A'&=\gr\left(\Uminus{w}_{\phi}\right)
\end{align*}
We now apply Corollary \ref{cor_samesize} which asserts 
$$\ell(w)=\rank(\tilde{G})+\ell(w')$$
$$\ell(u)=\rank(\tilde{G}_u)+\ell(u')$$
We have $\ell(w)=\ell(u)+1$. Moreover by inclusion $A'^0\supseteq A'^0_u$ and $\Phi^+(w')\supseteq \Phi(u')$ we have inequalities 
$\rank(\tilde{G})\geq \rank(\tilde{G}_u)$ and $\ell(w')\geq \ell(u')$. Hence we have two possible cases:
\begin{enumerate}
 \item[i)] $\rank(\tilde{G})=\rank(\tilde{G}_u)+1$ and $\ell(w')=\ell(u')$. The latter implies $w'=u'$, the former implies $\rank(\tilde{G})\gneq \rank(\tilde{G}_u)$. Hence the leading term of $F^\phi_{u(\alpha_i)}$ must be the new element $K^{-1}_\mu\in A'^0$. By orthogonality of $\supp(\phi)$, we find that $\mu=u(\alpha_i)$ is a new  root  in $\supp(\phi)=\{\mu\}\cup \supp(\phi_u)$. In particular, $w=us_i=s_\beta u$. This proves that $w'=u'$ is as asserted by Conjecture A as
 
 $$(\prod_{\beta\in\supp(\phi)}s_\beta)w
 =s_nu's_i
 =u'
 =w'$$
 \item[ii)] $\rank(\tilde{G})=\rank(\tilde{G}_u)$ and $\ell(w')=\ell(u')+1$. The former implies $\tilde{G}=\tilde{G}_u$ and $\supp(\phi)=\supp(\phi_u)$. The second implies together with $\Phi^+(w')\supset \Phi^+(u')$ that there exists a simple root $\alpha_j$ with $w'=u's_j$. Since the support is unchanged, the assertion of Conjecture A amounts to the assertion $\alpha_j=\alpha_i$. 
\end{enumerate}
\end{proof}

The second case is what we need to study further. A main example to have in mind is the following. It also illustrated that when we prolong from $u$ to $w=us_i$, the new degree $u'(\alpha_j)\in \Phi^+(w')$ in graded algebra  is \emph{not} directly the leading term of the new character shifted root vector ${F}_{u(\alpha_i)}^\phi$, as one might naively expect. Rather, the degree of this leading term might be already present in $\Phi^+(u')$, and the new degree comes from some sub-leading term. Let us discuss this more explicity:


\begin{example}\label{exm_A3}
Let $\g=A_3$ and $u=s_3s_1s_2$, then 
$$\Phi^+(u)=\{\alpha_1,\alpha_{3},\alpha_{123}\} \qquad \supp(\phi)=\langle \alpha_1,\alpha_3\rangle $$ 
where again we abbreviate $\alpha_{123}=\alpha_1+\alpha_2+\alpha_3$. The algebra $\Uminus{s_3s_1s_2}$ has a PBW basis in the generators
$$F_3,\qquad F_1=T_3^{-1}F_1,\qquad F_{123}:=T_3^{-1}T_1^{-1}F_2=[F_1,[F_3,F_2]_{q^{+1}}]_{q^{+1}}$$
In accordance with Conjecture A, the character shifted root vectors are
\begin{align*}
 F^\phi_{1}
 &=F_{1}+\phi(F_{1})K_{1}^{-1}\\
 F^\phi_{3}
 &=F_{3}+\phi(F_{3})K_{3}^{-1}\\
 F^\phi_{{123}}
 &=F_{123}
 +\phi(F_{1})[K_1^{-1},[F_3,F_2]_{q^{}}]_{q^{}}
 +\phi(F_{3})[F_1,[K_3^{-1},F_2]_{q^{}}]_{q^{}}\\
 &+\phi(F_{1})\phi(F_{3})[K_1^{-1},[K_3^{-1},F_2]_{q^{}}]_{q^{}}\\
&=F_{123}
 +\phi(F_{1})(q^{-1}-q)[F_3,F_2]_{q^{}}K_1^{-1}
 +\phi(F_{3})(q^{-1}-q)[F_1,F_2]_{q^{}}K_3^{-1}\\
 &+\phi(F_{1})\phi(F_{3})(q^{-1}-q)^2F_2 K_1^{-1}K_3^{-1}
 \end{align*}
These have highest degrees $0,\,0,\,\alpha_2$ and thus the graded algebra corresponds to the Weyl group element $u'=s_{1}s_{3}u=s_2$ as conjectured:
$$\gr\left(\Uminus{s_3s_1s_2}_{\phi}\right)\cong \Uminus{s_2}\langle K_{1}^{-1},K_{3}^{-1}\rangle$$

Now the  induction step is to prolong this by $\alpha_1$ to $w=us_{1}$, then we have a new root vector in degree $u(\alpha_1)=\alpha_{23}$, to be precise:
\begin{align*}
 F_{23}&:=T_3^{-1}T_1^{-1}T_2^{-1}F_1=T_3^{-1}F_2=[F_3,F_2]_{q^{+1}}\\
 F^\phi_{23}
 &=[F_3,F_2]_{q}
 +\phi(F_3)[K_3^{-1},F_2]_{q}\\
 &=F_{23}
 +\phi(F_3)(q^{-1}-q)F_2K_3^{-1}
 \end{align*}
The leading degree term $F_2$ of $F^\phi_{23}$ is already contained in $\Uminus{u'}$, as leading degree term of $F^\phi_{{123}}$. The new degree $u'(\alpha_j)$ in $\Uminus{w'}\supset \Uminus{u'}$ will be the leading degree of some suitable linear combination, and the main question is which will be the next-leading degree. Simply from degree considerations, there are two reasonable possibilities (this will be formalized as Criterion 1 below):
\begin{itemize}
 \item $u(\alpha_1)=u'(\alpha_j)=\alpha_{23}$, contrary to Conjecture A. Naively, this seems the natural choice because it is the next leading degree in $F^\phi_{23}$.
 \item $u(\alpha_1)+\alpha_{12}-\alpha_{23}=u'(\alpha_1)=\alpha_{12}$, according to Conjecture A. For this to happen, $F_2,{F}_{23}$ have to cancel simultaneously in a suitable linear combination of $F^\phi_{123},F^\phi_{23}$, and then ${F}_{12}$ from $F^\phi_{123}$ is the leading term.
\end{itemize}
Having explicit expressions for $F^\phi_{123},F^\phi_{23}$ we can check explicitly that the latter is the case, because the coefficients of $F_2,{F}_{23}$ in $F^\phi_{123},F^\phi_{23}$ are proportional:
\begin{align*}
 F^\phi_{123}-\phi(F_1)(q-q^{-1})F^\phi_{23}K_1^{-1}
 &=F_{123}+\phi(F_{3})(q^{-1}-q)[F_1,F_2]_{q^{}}K_3^{-1}
\end{align*}
\end{example}

\subsection{Criteria implying Conjecture A}

Let $C=\Uminus{w}_\phi$ with $w=us_i$ with $\ell(w)=\ell(u)+1$ and $u(\alpha_i)\not\in\supp(\phi)$ and let us assume that $\gr(\Uminus{u}_\phi)=\Uminus{u'}$ for $u'$ according to Conjecture A. The second case in Lemma \ref{lm_inductionstep} states that $\acf[\tilde{G}]\gr(\Uminus{w}_\phi)=\acf[\tilde{G}]\Uminus{w'}$ for $w'=u's_j$ with $\ell(w')=\ell(u')+1$ for some $\alpha_j\in \Pi$. Conjecture A holds for $w'$ iff we have $\alpha_j=\alpha_i$. 

We now collect three criteria when this is the case. They allows us to prove Conjecture A explicitly for given $\g$ of small rank or certain explicit $w\in W$, and they suffice to prove Conjecture A altogether for type $A_n$ and many cases in $B_n,C_n,D_n$. However, this approach is rather ad-hoc and a more systematic proof strategy would be much preferred. \\

As first criterion, we note the general relation between $\alpha_i$ and $\alpha_j$, which is potentially different. We find that this is a rather restrictive situation, on the level of root systems, and there is only a few ``critical cases'' up to reflection in which $\alpha_i\neq \alpha_j$ is possible, and which have to be treated further.

\begin{lemma}[Criterion 1]\label{lm_crit1}
Let $\acf[\tilde{G}]\gr(\Uminus{w}_\phi)=\acf[\tilde{G}]\Uminus{w'}$ for $w=us_i$ with $\ell(w)=\ell(u)+1$ and $w'=u's_j$ with $\ell(w')=\ell(u')+1$ for some $w,w',u,u'\in W$.
 \begin{enumerate}[a)]
 \item In general we have $u(\alpha_i)-u'(\alpha_j)\in\Z[\supp(\phi)]$.
 \item Let  $u(\alpha_i)-u'(\alpha_j)\in\Z[S]$ with $S$ a subset of $\supp(\phi)$ of minimal cardinality. Then one of the following cases holds:
 \begin{enumerate}
 \item[(i)] $\alpha_i=\alpha_j$ holds.
  \item[(ii)] $u(\alpha_i)=u'(\alpha_j)$ holds.
  \item[(ii)] $u(\alpha_i),u'(\alpha_j),S$ are up to reflection in a subsystem of rank~$2$.
 \item[(iv)] The tuple $u(\alpha_i),u'(\alpha_j),S$ is up to reflection in a parabolic subsystems ${\Phi_r\subset \Phi}$ of rank $r=3,4$ and is one of the following cases\\
 
 \begin{center}
 \begin{tabular}{l|rll}
 $\Phi_r$ & $u(\alpha_i)-u'(\alpha_j)$ & & $\in \Z[S]$\\
 \hline
 $A_3$ & $\alpha_{123}-\alpha_{2}$ & $=\alpha_1+\alpha_3$ & $\in \langle\alpha_1,\alpha_3\rangle_\Z$\\
 $B_3$ & $\alpha_{123}-\alpha_2$ & $=\alpha_1+\alpha_3$ & $\in \langle\alpha_1,\alpha_3\rangle_\Z$\\
 $B_3$ & $\alpha_{1123}-\alpha_2$ & $=\alpha_1+\alpha_1+\alpha_3$ &$\in \langle\alpha_1,\alpha_3\rangle_\Z$\\
 $C_3$ & $\alpha_{123}-\alpha_2$ & $=\alpha_1+\alpha_3$ & $\in \langle\alpha_1,\alpha_3\rangle_\Z$\\
 $D_4$ & $\alpha_{1234}-\alpha_4$ & $=\alpha_1+\alpha_2+\alpha_3$ & $\in \langle\alpha_1,\alpha_2,\alpha_3\rangle_\Z$\\
 \hline
 \end{tabular}
 \end{center}~
 
 \end{enumerate}
 \item Suppose that Conjecture A holds for $u$. Then either b) (i) applies, so Conjecture A holds for $w$ as well, or b) (iv) applies. Moreover, we have $u(\alpha_i)-u(\alpha_j)\in\Z[S]$, so also the tuple $u(\alpha_i),u(\alpha_j),S$ can be found in the same table. 
 \end{enumerate}
\end{lemma}
\begin{proof}
 \begin{enumerate}[a)]
 \item The proof of Lemma \ref{lm_inductionstep} shows that $\Uminus{w}_\phi$ contains a new element~$F^\phi_{u(\alpha_i)}$ and since $\phi$ is a character shifts with respect to $\supp(\phi)$, this algebra is still graded by cosets of $\Z[\supp(\phi)]$. On the other hand, the graded algebra $A'$ contains a new element in degree $u'(\alpha_j)$.
 \item Any set of $k$ roots can be reflected to a parabolic subsystem of rank at most $k$. Applying this assertion to $S$ and $\beta_i$ gives a subsystem $\Phi_r$ of rank $r=|S|+1$, which is moreover connected (otherwise the assumption can not hold even for a subset $S'$ smaller $S$). Applying the assertion a second time to $S$ gives a subset of $|S|=r-1$ many orthogonal simple roots. This can only be true for a subsystem of rank $2$ or for the cases $\Phi_r=A_3,B_3,C_3$ with $S$ reflected to the outmost simple roots $\alpha_1,\alpha_3$, or for $\Phi_r=D_4$ with $S$ reflected to the three outmost roots $\alpha_1,\alpha_2,\alpha_3$. The result then follows from directly inspecting these root systems for solutions $\beta-\gamma\in\Z[S]$. 
 \item This follows from $u(\alpha_i)-u(\alpha_j)\in\Z[S]$.
 \end{enumerate}
 \end{proof}

As a second criterion, we formulate the consequence of applying the induction step to two different presentations of $w$: 

\begin{definition}\label{def_unique}
 A Weyl group element $w$ has a \emph{unique ending} iff one of the following equivalent conditions applies:
 \begin{itemize}
 \item Any reduced expression of $w$ ends in the same letter $s_i$.
 \item $\Phi^+(w^{-1})$ contains a unique simple root $\alpha_i$.
 \item $w(\alpha_k)>0$ for all simple roots $\alpha_k$ except for one $\alpha_i$. 
 \end{itemize}
\end{definition}

\begin{lemma}[Criterion 2]\label{lm_crit2}
 Assume that $w$ has not a unique ending, i.e. $w=u_1s_i=u_2s_l$ with $\ell(w)=\ell(u_1)+1=\ell(u_2)+1$ and  $\alpha_i\neq \alpha_l$, and $u_1(\alpha_i),u_2(\alpha_l)\not\in\supp(\phi)$. Then if Conjecture A holds for $u_1,u_2$, then $\alpha_j=\alpha_i$, and thus Conjecture A holds also for $w$.
\end{lemma}
\begin{proof}
 As is well known from \cite[Thm. 29]{M64}, there is a series of braid group moves to transform the two presentations of $w$ into one another. In particular, under our assumptions, there exists a reduced expression of $w$ which ends
 $$w=r\cdots s_is_ls_i=r\cdots s_ls_is_l$$
 with $r^{-1}w$ of length $2,3,4,6$ depending on $(\alpha_l,\alpha_i),i\neq l$, such that a braid group move can be performed. Differently spoken
 $$u_1:=r\cdots s_is_l,\qquad w=u_1s_i$$
 $$u_2:=r\cdots s_ls_i,\qquad w=u_2s_l$$
with $\ell(w)=\ell(u_1)+1=\ell(u_2)+1$.  Assume that Conjecture A holds for the elements $u_1,u_2$ of shorter length. If the induction step Lemma \ref{lm_inductionstep} is applied to both reduced expressions $w=u_1s_i=u_2s_l$ in  
the second case of the Lemma, and using that $u_1(\alpha_i),u_2(\alpha_l)\not\in\supp(\phi)$ asserts that 
$u_1'=r'\cdots s_is_l$ and $u_2':=r'\cdots s_ls_i$
and that in particular
 $$\Phi(u'_1),\Phi(u'_2)\subset \Phi(w')$$
thus we get $\Phi(u'_2)\subset \Phi(u'_1s_{\alpha_j})$, and hence $\Phi(r'^{-1}u_2') =\Phi(\cdots s_ls_i)\subset \Phi(\cdots s_is_ls_{\alpha_j})=\Phi(r'^{-1}u_1's_{\alpha_j})$. Since $\ell(r'^{-1}u_2')\leq m_{il}-1$ the set $\Phi(r'^{-1}u_2')$ contains only one of the two simple roots $\alpha_i$ and $\alpha_l$ and the set $\Phi(r'^{-1}u_1')$ contains only the other one. If we have $\alpha_i\neq \alpha_j$ this directly leads to a contradiction.

\end{proof}

\subsection{The proof of Conjecture A for type \texorpdfstring{$A_n$}{An}}\label{sec_proofA}

We study the critical cases of Criterion 1: Assume $w,w',u,u',\alpha_i,\alpha_j$ are as in Lemma \ref{lm_inductionstep} such that the relation of $u(\alpha_i),\,u(\alpha_j)$ are as in Lemma \ref{lm_crit1} b) (iv) in the case labeled $A_3$ with $\alpha_i\neq \alpha_j$, contrary to Conjecture A. This means explicitly: We have a subset of the support $\Z[S]=\Z\sigma_1\oplus \Z\sigma_2$ for $\sigma_1,\sigma_2\in \Phi^+(w)$,
and there is a series of reflections that sends all relevant elements to a parabolic subsystem of type $A_2$ as described in the Lemma
$$\epsilon_1\sigma_1\mapsto \alpha_1,\quad 
\epsilon_2\sigma_2\mapsto \alpha_3,\quad
u(\alpha_i)\mapsto \alpha_{123},\quad
u(\alpha_j)\mapsto\alpha_{2},\quad
u'(\alpha_j)=\alpha_{123}$$
and it is not clear after reflections with which signs $\epsilon_1,\epsilon_2\in\{\pm 1\}$ the roots $\sigma_1$ and $\sigma_2$ appear. Note that this completely fixes the intersection of this $3$-dimensional space with the root system $\Phi$, because the remaining positive roots are $\sigma_1+u(\alpha_j)$ and $\sigma_2+u(\alpha_j)$. It also shows $(\alpha_i,\alpha_j)=0$. We pictorially write the positive roots of $A_3$ as 
\begin{center}
\begin{tabular}{ccccccccccc}
 &&$\alpha_{123}$&& \\
 &$\alpha_{12}$ && $\alpha_{23}$ & \\
 $\alpha_1$ && $\alpha_2$ &&  $\alpha_3$
\end{tabular}
\end{center}~\\
which correspond as described above to the following actual elements in the higher rank root system $\Phi$ (i.e. not just up to reflection)
\begin{center}
\begin{tabular}{ccccccccccc}
 &&$u(\alpha_i)=u'(\alpha_j)$&& \\
 &$\epsilon_1\sigma_1+u(\alpha_j)$ && $\epsilon_2\sigma_2+u(\alpha_j)$ & \\
 $\epsilon_1\sigma_1$ && $u(\alpha_j)$ &&  $\epsilon_2\sigma_2$
\end{tabular}
\end{center}~\\
 If we apply $u^{-1}$, then $u(\alpha_i),u(\alpha_j)$ are mapped to the simple roots $\alpha_i,\alpha_j$ in $\Phi$ and we now know already that $\alpha_i,\alpha_j$ are orthogonal. 
We observe $u^{-1}(\sigma_i)=s_i w^{-1}(\sigma_i)<0$ because by $\sigma_1,\sigma_2\in\Phi^+(w)$ implies $w^{-1}(\sigma_1),w^{-1}(\sigma_2)<0$ and this can only be changed to $>0$ by the subsequent application of $s_i$ if $\sigma_i= u(\alpha_i)$, which is not the case.
The images of the two unnamed roots in the middle row are the positive or negative roots
$$\gamma_k:=u^{-1}(\epsilon_k\sigma_k+u(\alpha_j))
=\epsilon_k u^{-1}(\sigma_k)+ \alpha_j$$
and from the $A_3$ root system we see
$$\gamma_1+\gamma_2=\alpha_i+\alpha_j$$
and we see the following property (as we see after applying  $u$, when $u(\alpha_i)-\epsilon_1\sigma_1-u(\alpha_j)=\epsilon_2\sigma_2$ and $u(\alpha_j)-\epsilon_1\sigma_1-u(\alpha_j)=-\epsilon_1\sigma_1$ and similarly for $1,2$ switched)
$$\alpha_i-\gamma_1 = \epsilon_2 u^{-1}(\sigma_2),\qquad \alpha_j-\gamma_1=-\epsilon_1u^{-1}(\sigma_1),$$
$$\alpha_i-\gamma_2=\epsilon_1 u^{-1}(\sigma_1),\qquad \alpha_j-\gamma_2=-\epsilon_2u^{-1}(\sigma_2).$$
These formulae imply severe restriction, of which we now discuss the first: Because $\alpha_i,\alpha_j$ are simple roots and the right-hand side is a (positive or negative) root, in the first formula either $\gamma_1>0,\epsilon_2=+1$ or $\gamma_1<0,\epsilon_2=-1$ and in the second formula either $\gamma_1>0,\epsilon_1=-1$ or $\gamma_1<0,\epsilon_1=+1$, which altogether implies $\epsilon_1=-\epsilon_2$. The third and fourth formula accordingly imply then $\gamma_2$ has the opposite sign as $\gamma_1$. We now choose without restriction of generality 
$$\gamma_1>0,\quad \epsilon_1=-1,\quad \gamma_2<0,\quad \epsilon_2=+1$$
which means our positive roots in the previous diagram read

\begin{center}
\boxed{
\begin{tabular}{ccccccccccc}
 &&$u(\alpha_i)=u'(\alpha_j)$&& \\
 &$-\sigma_1+u(\alpha_j)$ && $\sigma_2+u(\alpha_j)$ & \\
 $-\sigma_1$ && $u(\alpha_j)$ &&  $\sigma_2$
\end{tabular}
}
\end{center}~\\
After applying $u^{-1}$ we introduce the positive roots 
$$\gamma:=-\gamma_2,\quad \alpha_i+\gamma+\alpha_j=\gamma_1,\quad \sigma_k':=-u^{-1}(\sigma_k)$$
and by the formulae above $\alpha_i+\gamma=\sigma_1'$ and $\alpha_j+\gamma=\sigma_2'$. So after applying $u^{-1}$
and turning the picture so that the positive roots are on top we have 
\begin{center}
\boxed{
\begin{tabular}{ccccccccccc}
 &&$\alpha_i+\gamma+\alpha_j$&& \\
 &$\sigma_1'$ && $\sigma_2'$ & \\
 $\alpha_i$ && $\gamma$ &&  $\alpha_j$
\end{tabular}
}
\end{center}~\\
\begin{corollary}
The positive root $\gamma$ associated to our situation 
has the distinguished property 
\begin{align}\label{formula_distinguished}
    \alpha_i+\gamma\in\Phi^+,\qquad \gamma+\alpha_j\in \Phi^+
\end{align}
and in relation to the Weyl group element
$$
w(\gamma)=u(\sigma_1')=-u(\sigma_1),\quad 
w(\gamma+\alpha_i)=u(\gamma)=\sigma_k+u(\alpha_j),\quad
w(\gamma+\alpha_j)=u(\sigma_2')=-u(\sigma_2)
$$
\end{corollary}

\begin{example}
Recall Example \ref{exm_A3} for $A_3$, where  $u=s_3s_1s_2$, $u'=s_2$ and a support spanned by $\alpha_1,\alpha_3$ is prolonged by $\alpha_i=\alpha_1$ to $w=s_3s_1s_2s_1$ with either $w'=s_2s_1$ for $\alpha_j=\alpha_i$ (which is correct, according to Conjecture A, and confirmed by direct calculation) or with $w'=s_2s_3$ for $\alpha_j=\alpha_3$. In the latter case the picture above reads 
\begin{center}
\begin{tabular}{ccccccccccc}
 &&$u(\alpha_i)=u'(\alpha_j)=\alpha_{23}$&& \\
 &$-\sigma_1+u(\alpha_j)=\alpha_2$ && $\sigma_2+u(\alpha_j)=\alpha_{123}$ & \\
 $-\sigma_1=-\alpha_1$ && $u(\alpha_j)=\alpha_{12}$ &&  $\sigma_2=\alpha_3$
\end{tabular}
\end{center}~\\
and after application of $u^{-1}$ and turning the picture as above
\begin{center}
\begin{tabular}{ccccccccccc}
 &&$\alpha_i+\gamma+\alpha_j=\alpha_{123}$&& \\
 &$\sigma_1'=\alpha_{13}$ && $\sigma_2'=\alpha_{23}$ & \\
 $\alpha_i=\alpha_1$ && $\gamma=\alpha_2$ &&  $\alpha_j=\alpha_3$
\end{tabular}
\end{center}~\\
\end{example}

A canonical example of a root in any connected root system satisfying the distinguished property \eqref{formula_distinguished} is the following:

\begin{definition}
 For simple roots $\alpha_i,\alpha_j$ in a connected Dynkin diagram, there is a unique path from $\alpha_i$ to $\alpha_j$, and summing along the path without endpoints defines a root
 $$\rho_{ij}:=(\alpha_{i-1}+\cdots+\alpha_{j-1})$$
 with the property that $\rho_{ij}+\alpha_i,\rho_{ij}+\alpha_j\in\Phi$. Moreover it contains $\alpha_i$ and $\alpha_j$ with multiplicity zero.
\end{definition}
In this case we are finished, due to the following statement
\begin{corollary}\label{cor_Crit2}
Assume that $\gamma$ contains $\alpha_i$ with multiplicity zero, then Criterion 2 (Lemma \ref{lm_crit2}) applies. This showing Conjecture A in this case if it holds inductively for shorter words. 
\end{corollary}
\begin{proof}
    The Weyl group element $w$ has a unique ending $s_i$ iff $w(\alpha_k)>0$ for all $\alpha_k\neq \alpha_i$. But if $\alpha_i$ has multiplicity zero, this implies $w(\gamma)>0$, which is a contradiction to its definition giving $w(\gamma)=u(s_i(\gamma))=u(\sigma_1')=-\sigma_1<0$.
\end{proof}

Assume type $A_n$, then Criterion 1 (Lemma \ref{lm_crit1}) returns only critical cases of type $A_3$. Since in $A_n$ the \emph{only} element with the distinguished property \eqref{formula_distinguished}  is  $\rho_{ij}$ and anyhow all multiplicities are $\leq 1$, the root $\gamma$ with $\gamma+\alpha_i\in\Phi$ cannot contain $\alpha_i$ with multiplicity $\geq 1$, so Criterion 2 applies as discussed above. 

\begin{corollary}
 Conjecture A holds for type $A_n$.
\end{corollary}

\subsection{Remark on further cases}

We conclude by some remarks on type $B_n,C_n,D_n$. As we mentioned, a more structured approach would be desirable. Again Criterion~1 (Lemma \ref{lm_crit1}) lists critical case $A_3$ and $B_3$ or $C_3$ or $D_4$, which are the possible subdiagrams of the Dynkin diagram. We first assume again case $A_3$ and list the distibuished roots of $\gamma$ explicitly:
\begin{fact}
 By directly inspecting the root systems $B_n,C_n,D_n$, we find that for given $\alpha_i,\alpha_j$ the only positive roots $\beta$ with $\beta+\alpha_i,\beta+\alpha_j\in \Phi$ are
 \begin{align*}
 \gamma & =\rho_{ij} \\
 \gamma & =\alpha_{i-1}+\cdots+\alpha_j+2\alpha_{j-1}+\cdots\\
 \gamma & =\alpha_{j-1}+\cdots+\alpha_i+2\alpha_{i-1}+\cdots
 \end{align*}
 where the simple roots are written in descending order and $\alpha_1$ is the unique short simple root for $B_n$, resp. $\alpha_1$ is the unique long simple root for $C_n$, resp. $\alpha_1,\alpha_2$ are the two short legs for $D_n$, and where the final summands $\cdots$ depend on the type of the root system.
\end{fact}
In the first and second cases again Corollary \ref{cor_Crit2} applies,  because $\gamma$ contains $\alpha_i$ with multiplicity zero. The third case is more complicated and requires by-hand arguments.\\

The critical cases $B_3,C_3$ in $B_n,C_n$ can again be analyzed with Corollary \ref{cor_Crit2}, because by direct inspection of the root system we only find very few possibilities for $\beta$. The critical case $D_4$ in $D_n$ seems again to be harder.\\

\subsection{Remark on direct computations}

A by-hand method we have explored is the direct computation of character shifts in particularly nice cases: In general, it is hard to relate the coproduct of a reflected element to the coproduct of the original element, but we can obtain enough information under a multiplicity-zero condition. It enters through the following formulae:

\begin{proposition}\label{prop_explicitShift}
 Let $\alpha_m$ be a simple root and consider $T_{v^{-1}}^{-1}(F_i)$ for any $v\in W, i\in I$, which is in degree $\nu:=v(\alpha_i)$ and $\ell(vs_i)=\ell(v)+1$.
 Assume that $(\alpha_m,\nu)=-1$ and that $\alpha_m$ has multiplicity zero in the degree of $v(\alpha_i)$. Let $\phi_m$ be a character on $\Uminus{s_mvs_i}$. Then 
$$(T_{s_m}^{-1}T_{v^{-1}}^{-1}(F_i))_{\phi_m}=
 T_{s_m}^{-1}T_{v^{-1}}^{-1}(F_i)
 +(q_m-q_m^{-1}) \phi_m(F_m)\cdot T_{v^{-1}}^{-1}(F_i)K_m^{-1}$$
 If $(\alpha_m,\nu)=-c$ a similar formula holds. 
 \end{proposition}~
 
\begin{proof}~
 
 To compute the character shift $(T_{s_m}^{-1}T_{v^{-1}}^{-1}(F_i))_{\phi_m}$ we need to relate coproduct and reflection, which is involved. From \cite[Prop. 37.3.2]{Lusz93}  we have the formula\footnote{We use the notation $T_{s_m},T_{s_m}^{-1}$ in \cite{Jan96}, which translates by the remark on p. 146 to Lusztig's notation $T_{m,1}'',T_{m,{-1}}'$.} 
 \begin{align*}
 &(T_{s_m}^{-1}\otimes T_{s_m}^{-1})\Delta(T_{s_m}(x))\\
 &=\underbrace{\left(\sum_{l=0}^\infty q_m^{l(l-1)/2}\{l\}_{q_m} F_m^{(l)}\otimes E_m^{(l)}\right)}_{L'}
 \Delta(x)
 \underbrace{\left(\sum_{r=0}^\infty (-1)^r q_m^{-r(r-1)/2}\{r\}_{q_m} F_m^{(r)}\otimes E_m^{(r)}\right)}_{L''} 
 \end{align*}
 where $\{l\}_{q^m}=[l]_{q_m}(q_m-q_m^{-1})^l$. We need an analogous formula for $T_{s_m}^{-1}$, by substituting $x=T_{s_m}^{-1}(y)$ and using $L'L''=1$ from \cite[ Prop. 5.3.2]{Lusz93}
 \begin{align*}
 &\Delta(T_{s_m}^{-1}(y))\\
 &=\left(\sum_{r=0}^\infty (-1)^r q_m^{-r(r-1)/2}\{r\}_{q_m} F_m^{(r)}\otimes E_m^{(r)}\right) 
 (T_{s_m}^{-1}\otimes T_{s_m}^{-1})\Delta(y)
 \left(\sum_{l=0}^\infty q_m^{l(l-1)/2}\{l\}_{q_m} F_m^{(l)}\otimes E_m^{(l)}\right)
 \end{align*}
 If we apply the character $\phi_m\otimes 1$, then the only nonzero contributions are those tensor summands of $(T_{s_m}^{-1}\otimes T_{s_m}^{-1})\Delta(y)$ and hence $\Delta(y)$ where the left tensor factor has degree in $\Z\alpha_m$. \\
 
 We now use our assumption that $y:=T_{v^{-1}}^{-1}(F_i)$ in degree $v(\alpha_i)$ does not contain a factor $F_m$. Thus in this case the only nonzero contribution comes from the summand $1\otimes T_{s_m}^{-1}T_{v^{-1}}^{-1}(F_i)$ in $(T_{s_m}^{-1}\otimes T_{s_m}^{-1})\Delta(y)$, thus in this case
 \begin{align*}
 &(\phi_m\otimes 1)\Delta(T_{s_m}^{-1}(y))\\
 &=
 \sum_{r,l=0}^\infty
 (-1)^r q_m^{-r(r-1)/2}\{r\}_{q_m}
 q_m^{l(l-1)/2}\{l\}_{q_m}\cdot 
 \phi_m(F_m^{(r)}F_m^{(l)})\cdot 
 E_m^{(r)}\left(T_{s_m}^{-1}T_{v^{-1}}^{-1}(F_i)\right)E_m^{(l)}\\
%
 &=
 \sum_{s=0}^\infty
 q_m^{s(s-1)/2}(q_m-q_m^{-1})^s 
 \phi_m(F_m)^s
 \sum_{r=0}^s
 (-1)^r q_m^{(1-s)r}\cdot 
 E_m^{(r)}\left(T_{s_m}^{-1}T_{v^{-1}}^{-1}(F_i)\right)E_m^{(s-r)}\\
 \end{align*}
 By reasons of degree (or by the quantum Serre relation), all terms $s\neq 0,1$ vanish for $(\alpha_m,\nu)=-1$. If $(\alpha_m,\nu)=-c$, then similarly only the terms up to $s=c$ are nonzero. \\
 
 The last equality can then be most easily calculated by hand, since by the assumption on the multiplicity of $\alpha_m$, we have that $T_{-1}^{-1}(F_i)$ is a linear combination of terms $F'F_kF''$ where $F',F''$ are products of $F_l$ with $(\alpha_m,\alpha_l)=0$. Then by the multiplicativity of $T_{s_m}^{-1}$ and the defining formulae \cite[8.14 (8')]{Jan96}  we have 
 $$T_{s_m}^{-1}T_{v^{-1}}^{-1}(F_i)=T_{s_m}^{-1}(F'F_kF'')=F'(T_{s_m}^{-1}(F_k))F''=F'[F_m,F_k]_{q_m}F''$$
 and similarly by $[E_m,F_l]=[K_m,F_l]=0$ we get further
 $$[E_m,(T_{s_m}^{-1}T_{v^{-1}}^{-1}(F_i)]
 =F'[E_m,[F_m,F_k]_{q_m}]F''
 =F'F_kK_m^{-1}F''=F'F_kF''K_m^{-1}$$
 This proves the assertion. 
\end{proof}

\begin{lemma}[Criterion 3]\label{lm_crit3}
 Let $w=us_i$ with $\ell(w)=\ell(u)+1$ be a Weyl group element and $u=s_m s_n u'$ with $\ell(u)=\ell(u')+2$ and a character $\phi$ with support containing $\{\alpha_m,\alpha_n\}$. 
 Assume that the simple root $\alpha_m$ appears in $u(\alpha_i)-\alpha_m\in\Phi$ with multiplicity zero. Then Conjecture A holds for $w$, if it holds for the shorter elements $u,s_nu,s_mu$.
 \end{lemma}
\begin{proof}
 We want to explicitly construct an element in $\Uminus{s_ms_nu's_i}$ which has a character shift with leading term $T_{u'^{-1}}^{-1}(F_i)$ as asserted, using the inductive assumption on $u=s_ms_nu'$ as well as on $s_n u' s_i$ and on $s_mu's_i$ for smaller support $\{\alpha_n\}$ resp. $\{\alpha_m\}$. For the third inductive assumption we require the stronger explicit assertion from Proposition \ref{prop_explicitShift} for $v=u'$ that relies on the assumption of multiplicity one:
 $$(T_{s_m}^{-1}T_{u'^{-1}}^{-1}(F_i))_{\phi_m}=
 T_{s_m}^{-1}T_{u'^{-1}}^{-1}(F_i)
 +(q_m-q_m^{-1}) \phi_m(F_m)\cdot T_{u'^{-1}}^{-1}(F_i)K_m^{-1}$$
 Now the proof proceeds as follows:
 By the inductive assumption on $s_n u' s_i$ there exists an $X\in \Uminus{s_n u' s_i}$ such that we have ${X^{\phi_n}}={T_{u'^{-1}}^{-1}(F_i)}$ in $\gr \Uminus{s_n u' s_i}_{\phi_n}$, or differently spoken the character shift has leading term 
 $$X^{\phi_n}=T_{u'^{-1}}^{-1}(F_i)+\cdots$$
 We may assume $X$ to be chosen in degrees $u'(\alpha_i)+\Z\alpha_n$, because $\Uminus{s_n u' s_i}_{\phi_n}$ is still graded by $\Z\alpha_n$-cosets. \newline 
 We now apply $T_{s_m}^{-1}$, which commutes with the $\phi_n$ character shift, so we have constructed an element $T_{s_m}^{-1}(X)\in \Uminus{s_ms_n u' s_i}$ with leading term $$(T_{s_m}^{-1}X)^{\phi_n}=T_{s_m}^{-1}T_{u'^{-1}}^{-1}(F_i)+\cdots$$
 because the only degree, which is larger then $u'(\alpha_i)$ and becomes smaller after reflection, is $s_mu'(\alpha_i)$, which is not in $u'(\alpha_i)+\Z\alpha_n$. \newline
 We now apply a $\phi_m$ character shift to both sides of the formula, giving 
 $$(T_{s_m}^{-1}X)^{\phi_m,\phi_n}=\left(T_{s_m}^{-1}T_{u'^{-1}}^{-1}(F_i)\right)_{\phi_m}+\cdots$$
 By Proposition \ref{prop_explicitShift} we know the right-hand side shift explicitly, so
 $$(T_{s_m}^{-1}X)^{\phi_m,\phi_n}=(q_m-q_m^{-1}) \phi_m(F_m)\cdot T_{u'^{-1}}^{-1}(F_i)K_m^{-1}+\cdots$$
 which concludes the proof.\\
 \end{proof}
 
 We remark that without multiplicity zero we cannot control on which expression we have to apply the inductive assumption in order to get precisely the root vector in the last line. For example, the character shift could be zero or (since we apply $T_m^{-1}$) the initial element could even have to be chosen to be outside $U^-$. \\
 
 We remark further, that in terms of degrees, the following is hidden in the inductive assumptions of the proof, compare Example \ref{exm_A3}: The new PBW generator in $\Uminus{s_nu's_i}$ is $T^{-1}_{s_n}T^{-1}_{u'^{-1}}(F_i)$, which is in degree $s_nu'(\alpha_i)=\beta$, while the new degree in the graded algebra is $u'(\alpha_i)=\beta+\alpha_n$. This is only possible, because (apparently, by induction) some element $Y$ in degree $\beta+\alpha_n=u'(\alpha_i)=s_nu'_1(\alpha_l)$ is already present in $\Uminus{s_nu'}$, with character shift $Y^{\phi_n}=Y+T_{s_n}^{-1}T^{-1}_{u'^{-1}}(F_i)+\cdots$. Together with $T^{-1}_{s_n}T^{-1}_{u'^{-1}}(F_i)$ this makes possible the linear combination $X=Y-T^{-1}_{s_n}T^{-1}_{u'^{-1}}(F_i)$ in mixed degrees, where the leading terms in the character shifts cancel and the next leading term is in degree $u'(\alpha_i)=\beta+\alpha_n$.

\section{The graded algebra determining the representation theory}

\subsection{Conjecture B}

Our ultimate hope is that the description of the graded algebra in Conjecture A and Corollary \ref{cor_conjA}, which we have proven in many cases, 
can determine for a triangular right coideal subalgebras of the form in  Corollary \ref{cor_triangular}, for which combinations of data this gives actually a basic and a Borel subalgebra:

\begin{conjectureB}\label{conj_B}
 A right coideal subalgebra $C$ of the form given in Corollary \ref{cor_triangular}
 $$\Uminus{w_-}_{\phi_-}\acf[L]S(\Uplus{w_+})_{\phi_+}$$ 
 is a Borel subalgebra iff $w_+w_-'^{-1}=w_0$ with $\ell(w_+)+\ell(w_-'^{-1})=\ell(w_0)$, where as in Conjecture A 
 $$w_-'=(\prod_{\beta\in\supp(\phi)} s_\beta)w_-$$
\end{conjectureB}
Note that then the localization of the graded algebra is as in Corollary \ref{cor_conjA}
 $$\Uminus{w_-'}_{\phi_-}\acf[\tilde{G}]S(\Uplus{w_+})_{\phi_+}$$ 
\begin{example}
 In the homogeneous case $\phi_+=\phi_-=\varepsilon$ where $w_-'=w_-$ the condition is clear: If $\ell(w_+)+\ell(w_-'^{-1})>\ell(w_+w_-'^{-1})$ then there are common roots $\alpha,-\alpha$ in $\Phi^+(w_+),\Phi^-(w_-)$, producing a full quantum subgroups $U_q(\sl_2)$, which is surely not basic. If $\ell(w_+w_-'^{-1})<\ell(w_0)$ then there are other basic algebras with larger $w_+,w_-$.
\end{example}
 

The conjecture has essentially two parts, which are both
representation-theoretic: 
\begin{itemize}
\item Proving that $\ell(w_-'^{-1}w_+)<\ell(w_-')+\ell(w_+)$
implies non-basic 
\item Proving that $\ell(w_-'^{-1}w_+)=\ell(w_-')+\ell(w_+)$ implies basic. 
\end{itemize}
From these two it would then follow that Conjecture~B characterizes basic right coideal
subalgebras that are
at least maximal among all triangular basic right coideal subalgebras $C$ with $C\cap U^0$ a Hopf algebra.

We now prove the first direction in general using Conjecture A, in all cases where Conjecture A is proven.

The second part we can so far only prove in special cases like $A_1,A_2$ below
or $\supp(\phi_+)=\supp(\phi_-)=\Phi^+(w_+)\cap\Phi(w_-)$ by explicitly constructing the algebra. It would be desirable to find a proof using again the graded algebra and 
Conjecture~A.

\subsection{Proof of Conjecture B in one direction}

We will now prove Conjecture~B in one direction, by explicitly constructing irreducible 
higher-dimensional representations as composition factors of restricted
$U$-representations. For example, this result allows us to prove that
certain basic right coideal subalgebras we constructed \cite{LV17} are maximal at least
 among the \emph{triangular} basic algebras $C$ with $C\cap U^0$ a Hopf algebra.

\begin{lemma}\label{Annichtede}
 For $w_-\in W$ assume Conjecture A holds and let $w_-'$ be accordingly. Let $w_+\in W$. Let $C$ be a right coideal subalgebra of the form in Corollary \ref{cor_triangular}
 $$C=\Uminus{w_-}_{\phi^-}\acf[L] S(\Uplus{w_+})_{\phi_+}$$
 with $L\subset (\supp(\phi^+)\cap \supp(\phi^-))^{\bot}$. Then: If $\ell(w_-'^{-1}w_+)<\ell(w_-')+\ell(w_+)$ then $C$ is not basic.
\end{lemma}
\begin{proof}

For an arbitrary root system $\ell(w_-'^{-1}w_+)<\ell(w_-')+\ell(w_+)$ implies $\Phi^+(w_+)\cap\Phi^+(w_-')\neq\emptyset$,
so there is a root $\mu\in\Phi^+(w_+)\cap\Phi^+(w'_-)$. By Conjecture A, this assumption about the graded algebra $\gr(C)$ means that in $C$ there are elements $E,F$ with leading terms with respect to the $\mathbb{Z}$-grading the root vectors 
$E_{\mu}K_{\mu}^{-1}$ and $F_{\mu}$. Note that while $E=E_\mu$ the element $F$ starts usually with a higher root and is not even necessarily a character shifted root vector.\\

Our general strategy to construct irreducible representations of dimension
${>1}$ is as follows: We take some suitable finite-dimensional $U$
representation $V=L(\lambda)$, which we want to restrict to $C$ and decompose
into irreducible composition factors, and we wish to prove not all of them are
one-dimensional. The action of the commutator $[E,F]_1\in C$ on every one-dimensional $C$-module is trivial (because any commutator acts trivially on any $1$-dimensional representation) so if the composition series of $C$ only contains such representations, then $[E,F]_1$ acts nilpotently on $V$. 

But $E$ contains only terms in degree $\leq\mu$ and $F$ only terms in degree $\geq\mu$, so the commutator acts as a lower triangular matrix with diagonal entries $[E_{\mu}K_{\mu}^{-1},F_{\mu}]_1$. So we need to find a $U$ module $V=L(\lambda)$ with non-zero eigenvalues for $[E_{\mu}K_{\mu}^{-1},F_{\mu}]_1$, then this proves the existence of higher-dimensional irreducible composition factors.\\

Claim: For any root $\mu$ there exists a finite-dimensional $U$ representation $V$ on which the commutator $[E_\mu,F_\mu]$ has a non-zero eigenvalue i.e. does not act nilpotently. For $\mu$ simple this is easily seen from the highest-weight vector in $V=L(\lambda)$ for any $(\lambda,\mu)\neq 0$ (and for $A_n$ in general using minuscule weights). We generalize this approach as follows:\\
 

Let $\mu=w(\alpha_i)$ and choose any weight $\lambda$ with $(\alpha_i,\lambda)<0$. Then in the irreducible $U$-module $L(\lambda)$ the weight-spaces $w(\lambda)$ and $w(\lambda-\alpha_i)$ are both one-dimensional with basis $T_w(v_\lambda),T_w(F_{\alpha_i}v_\lambda)$ since they are reflections of one-dimensional weight spaces. Since $E_\mu=T_w E_{\alpha_i}$ and $F_\mu=T_{w^{-1}}^{-1}F_{\alpha_i}$ we can evaluate the commutator $[E_\mu,F_\mu]$ on $v:=T_w v_\lambda$, but things are slightly complicated by $T^{-1}_{w^{-1}}$ on $F$:
\begin{align*}
 [E_\mu,F_\mu]v
 &=(T_w E_{\alpha_i})(T_{w^{-1}}^{-1}F_{\alpha_i})(T_w v_\lambda)
 -(T_{w^{-1}}^{-1}F_{\alpha_i})(T_w E_{\alpha_i})(T_w v_\lambda)\\
 &=T_w E_{\alpha_i}T_w^{-1}T_{w^{-1}}^{-1}F_{\alpha_i}T_{w^{-1}}T_w v_\lambda -0
\end{align*}
where the second term vanishes by construction since $v_\lambda$ is a highest-weight-vector and $(T_w E_{\alpha_i})(T_w v_\lambda)=T_wE_{\alpha_i}v_\lambda=0$. Now since the weight-spaces $w(\lambda)$ and $w(\lambda-\alpha_i)$ are both one-dimensional and the Lusztig-automorphisms are bijective, we have scalars $a,b\neq 0$ with 
$$T_{w^{-1}}T_w v_\lambda=av_\lambda
\qquad
T_w^{-1}T_{w^{-1}}^{-1} (F_{\alpha_i}v_\lambda)= b(F_{\alpha_i}v_\lambda)$$
Since we further have 
$$E_{\alpha_i}F_{\alpha_i}v_\lambda
=[E_{\alpha_i},F_{\alpha_i}]v_\lambda
=\frac{K-K^{-1}}{q-q^{-1}}v_\lambda
=\frac{q^{(\alpha,\lambda)}-q^{-(\alpha,\lambda)}}{q-q^{-1}} v_\lambda
$$
it follows as claimed that $v=T_w v_\lambda$ is a non-zero eigenvector
$$[E_\mu,F_\mu]v=\frac{q^{(\alpha,\lambda)}-q^{-(\alpha,\lambda)}}{q-q^{-1}}\cdot a\cdot b\cdot v_\lambda$$
Thus as discussed above: Since $[E_\mu,F_\mu]$ acts not nilpotently on the
finite-dimensional $U$-module $L(\lambda)$, the restriction of this module
to $C$ containing as leading term $E_\mu,F_\mu$ cannot only have one-dimensional
composition-factors. Thus we found a higher-dimensional irreducible composition
factor and $C$ is not basic.
\end{proof}

\section{Example \texorpdfstring{$A_2$}{A2}}
We have already determined all Borel subalgebras of $U_q(\sl_2)$ in Section \ref{sec_exampleA1}. We conclude our article by treating the next case $U_q(\sl_3)$ explicitly. In particular, we give a completed classification of \emph{all} Borel subalgebras in this case by hand, so that our conjectures and their impact can be checked against a first realistic example.\\

Any triangular right coideal subalgebra  $C$ with $C\cap U^0$ a Hopf algebra is by Corollary \ref{cor_triangular} of the form 
$$C=\Uminus{w_-}_{\phi_-}\acf[L]S(\Uplus{w_+})_{\phi_+}$$
where $L\perp \supp(\phi_+)\cap\supp(\phi_-)$ and 
 the supports of the characters $\phi_+,\phi_-$ consist of mutually
orthogonal roots. We have proven Conjecture A in this case, and the proven direction of Conjecture B states that 
$\ell(w_-'^{-1}w_+)<\ell(w'_-)+\ell(w_+)$ implies that $C$ is not basic, and we have conjectured that the triangular Borel subalgebras are precisely those with $w_-'^{-1}w_+=w_0$. 

From certain considerations we would expect moreover that in general for a trinagular Borel subalgebra $\supp(\phi_+)=\supp(\phi_-)$ and $L=(\supp(\phi_+)\cap\supp(\phi_-))^\perp$.\\

These expectations would lead to the following candidates:
\begin{itemize}
 \item $\supp(\phi_+)=\supp(\phi_-)=\{\}$ i.e. $\phi_+,\phi_-$ trivial. These are the homogeneous Borel subalgebras and we have already discussed them in Section \ref{sec_standardBorel}. They are reflections of $U^-$, explicitly
 $$\Uminus{w_-}U^0S(\Uplus{w_+}),\qquad w_-^{-1}w_+=w_0$$
 \item $\supp(\phi_+)=\supp(\phi_-)=\{\alpha_1\}$, then in particular we must have $\alpha_1\in \Phi^+[w_\pm]$ leaving the three cases $w_\pm=s_1,s_1s_2,s_1s_2s_1$ with $w_\pm'=s_1w_\pm=1,s_2,s_2s_1$. The relation $w_-'^{-1}w_+=w_0$ leaves the following three cases:\\
\begin{center}
 \begin{tabular}{ccc|ccc}
 $w_-$ & $w'_-$& $w_+$ & $\Phi^+(w_-)$ & $\Phi^+(w'_-)$ & $\Phi^+(w_+)$ \\ 
 \hline
 $s_1$ & $1$ & $s_1s_2s_1$ & $\{\alpha_1\}$ & $\{\}$ & $\{\alpha_1,\alpha_2,\alpha_{12}\}$ \\
 $s_1s_2$ & $s_2$ & $s_1s_2$ & $\{\alpha_1,\alpha_{12}\}$ & $\{\alpha_2\}$ & $\{\alpha_1,\alpha_{12}\}$ \\
 $s_1s_2s_1$ & $s_2s_1$ & $s_2$ & $\{\alpha_1,\alpha_2,\alpha_{12}\}$ & $\{\alpha_2,\alpha_{12}\}$ & $\{\alpha_1\}$ \\
 \hline\\
 \end{tabular}
 \end{center}
\item $\supp(\phi_+)=\supp(\phi_-)=\{\alpha_2\}$ leaves only cases that are isomorphic to the former ones by diagram automorphism $\alpha_1\leftrightarrow\alpha_2$.
 \item $\supp(\phi_+)=\supp(\phi_-)=\{\alpha_{12}\}$ leaves the three cases
		$w_\pm=s_2s_1,s_1s_2,s_1s_2s_1$ with $w_\pm'=s_{12}w_\pm=s_2,s_1,1$. The
		relation $w_-'^{-1}w_+=w_0$ leaves up to diagram automorphism the case
		$w_-=s_1s_2,w'_-=s_1,w_+=s_2s_1$.
\end{itemize}
%

All candidates above except the homogeneous Borel subalgebras and except the second column $w_-=s_1s_2,\, w_-'=s_2,\,w_+=s_1s_2$, which we treat afterward in Section \ref{typ3insl3}, are of the simplest type of a possible Borel subalgebra, which we have studied in \cite[Sec. 3]{LV17}: We say that a triangular right coideal subalgebra $C$ of the form in Corollary~\ref{cor_triangular} has \emph{full support} iff
\begin{align}\label{formula_fullsupport}
    \Phi^+(w_+)\cap\Phi^+(w_-)=\supp(\phi_+)=\supp(\phi_-)
\end{align}
In this case, $C$ is essentially a products of commuting quantum Weyl algebras, filled up as much as possible with the standard Borel subalgebra. In \cite[Thm. 3.25]{LV17} (depending on a conjecture which we prove for type $A_n$) we prove that every Borel subalgebra of this type is a reflections of the case $w_+=w_0,\,w_-=\prod_{\beta\in \supp(\phi_-)}s_\beta$ (note that reflections of coideal subalgebras do not necessarily produce coideal subalgebras). Conversely, our main result in \cite[Thm. 3.3]{LV17} states that $C$ is indeed in general a Borel subalgebra. We now discuss this Borel subalgebra explicitly:

\subsection{The Borel subalgebra \texorpdfstring{$\Uminus{s_1}_{\phi_-}\acf[(K_1K_2^2)^{\pm 1}] S(\Uplus{s_1s_2s_1})_{\phi_+}$}{w=s1,s1s2s1}}

This right coideal subalgebra $C$ with full support \eqref{formula_fullsupport} is generated as an algebra by 
$$F^\phi_1=F_{1}+\lambda' K_{1}^{-1},
\qquad (K_1K_2^2)^{\pm 1},
\qquad E^\phi_1=E_{{1}}K_{{1}}^{-1}+\lambda K_{1}^{-1} 
\qquad E_{2}K_{2}^{-1}$$ 
where $\phi^+(E_{\alpha_1}K_{\alpha_1}^{-1})=\lambda$ and $\phi^-(F_{\alpha_1})=\lambda'$ and $0$ else, such that $\lambda\lambda'=\frac{q^2}{(1-q^2)(q-q^{-1})}$.\\
\begin{figure}[h]
 \includegraphics[scale=.2]{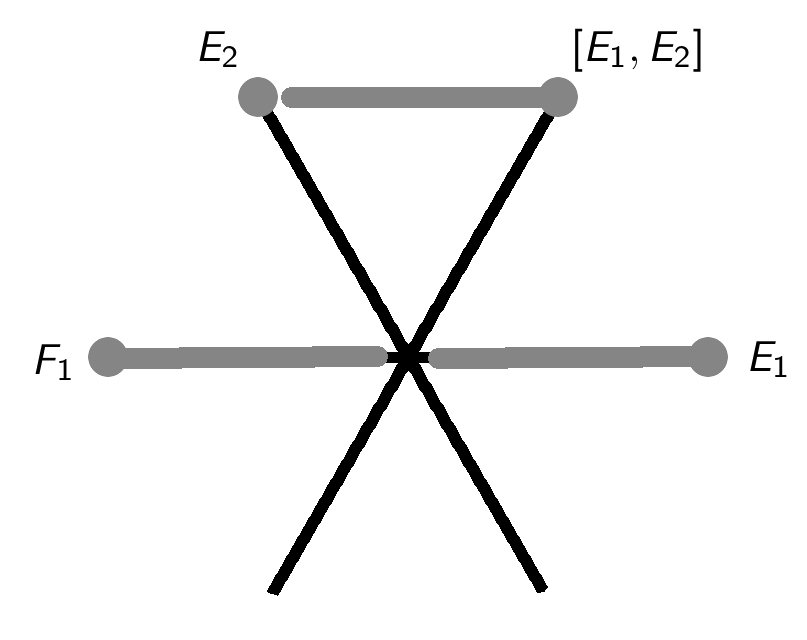}
 \caption{Picture of $\Phi^+(w_\pm)$ with gray lines indicating character shifts. The contained Weyl algebra is  visible on the $X$-axis.}
\end{figure}

In the $\Z$-grading, these elements have degrees $0,0,\alpha_1,\alpha_2$ and the
graded algebra is in accordance with the proven Conjecture A
$\gr(C)=\acf[(K_1K_2^2)^{\pm 1},K_1^{-1}] S(\Uplus{s_2s_1}$. 
The right coideal
subalgebra $C$ contains a Weyl algebra generated by $F^\phi_1,E^\phi_1$ from type $A_1\subset A_2$, and this injection splits via the algebra map sending $E_2K_2^{-1}\mapsto 0$. \\

We now prove that this algebra is basic: A direct calculation (which holds in great generality, see \cite[Lm. 5.13]{LV17}) shows that for any element $X_\mu\in S(\Uplus{s_1s_2s_1})_{\phi_+}$ in degree $\mu$ holds 
$$[F^\phi_1,X_\mu]_{q^{(\alpha_1,\mu)}}=0$$
Thus for any finite-dimensional representation $V$ we can consider the subspace where $E_{2}K_{2}^{-1}$ acts by zero and this is again a $C$-representation. Hence for irreducible $V$ the element $E_{2}K_{2}^{-1}$ acts by zero and the action of our algebra factorizes through the Weyl algebra in Example \ref{Weyl algebra}.\\

This algebra is the smallest example of a large family of basic triangular right coideal subalgebras that we constructed in \cite{LV17} for type $A_n$. 

\subsection{The Borel subalgebra \texorpdfstring{$\Uminus{s_1s_2}_{\phi_-}\acf[(K_1K_2^2)^{\pm 1}] S(\Uplus{s_1s_2})_{\phi_+}$}{w=s1s2,s1s2}}\label{typ3insl3}

This more complicated right coideal subalgebra $C$, which does not have full support \eqref{formula_fullsupport}, is generated as an algebra by 
\begin{align*}
E^\phi_{1}&:=E_{1}K_{1}^{-1}+\lambda K_{1}^{-1}\\
F^\phi_{1}&:=F_{1}+\lambda' K_{1}^{-1}\\
K^{\pm 1}&:=(K_1K_2^2)^{\pm1}\\
E^\phi_{12}&:=E_{12}(K_1K_2)^{-1}+(1-q^{-2})\lambda E_{2}(K_1K_2)^{-1}\\
F^\phi_{12}&:=F_{12}+(q^{-1}-q)\lambda' F_{2}K_{1}^{-1}
\end{align*}

\begin{figure}[h]
 \includegraphics[scale=.2]{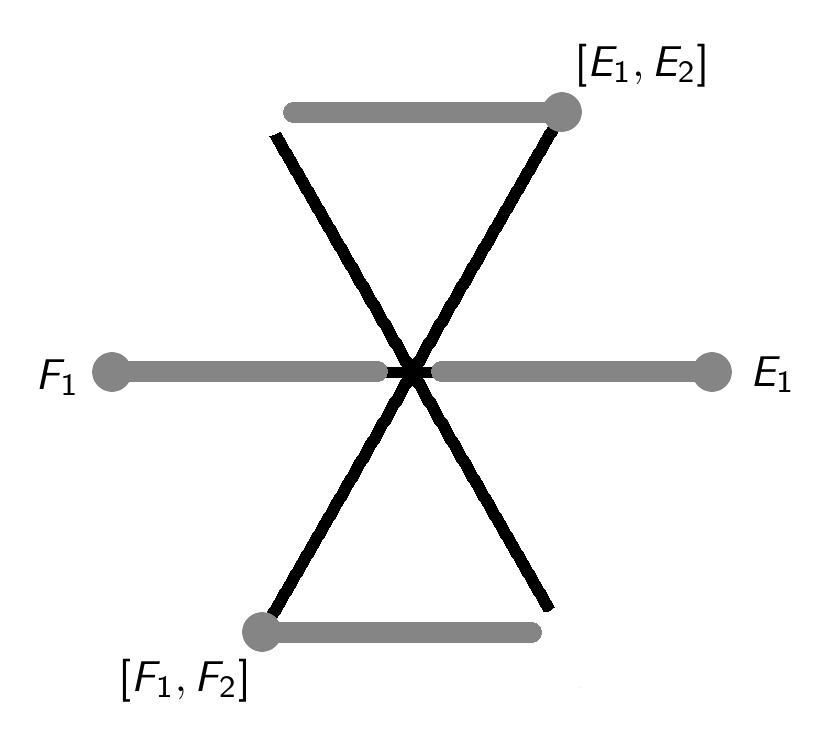}
 \caption{Picture of $\Phi^+(w_\pm)$ with gray lines indicating character shifts. We have two Weyl algebras, one extending another.}
\end{figure}

The graded algebra is in accordance with Conjecture A 
$$\gr(C)=\Uminus{s_2}\acf[(K_1K_2^2)^{\pm 1},K_1^{-1}] S(\Uplus{s_1s_2})$$
Here again the characters are given by 
$\phi^+(E_{\alpha}K_{\alpha}^{-1})=\lambda$ and $\phi^-(F_{\alpha})=\lambda'$ and $0$ else, such that $\lambda\lambda'=\frac{q^2}{(1-q^2)(q-q^{-1})}$. The product of the constants $c_1:=(1-q^{-2})\lambda$ and $c_2:=(q^{-1}-q)\lambda'$ appearing in the relations is 
$$c_1c_2=(1-q^{-2})\lambda(q^{-1}-q)\lambda'= (1-q^{-2})(q^{-1}-q)\frac{q^2}{(1-q^2)(q-q^{-1})}=1$$ 
We calculate the commutator relation 
\begin{align*}
 &[E_{12}K_{12}^{-1}+c_1 E_{2}K_{12}^{-1},F_{12}+c_2 F_{2}K_{1}^{-1}]_{q^2}\\
 &= (F_{1}+\lambda'K_{1}^{-1})(E_{1}K_{1}^{-1}+\lambda K_{1}^{-1})(q^4-q^2)+\frac{q^4}{q-q^{-1}}1
\end{align*}
Thus the commutators of all generators of $C$ are 
\begin{align}
 [K,E^\phi_{1}]_1&=[K,F^\phi_{1}]_1=[K,E^\phi_{12}]_1=[K,F^\phi_{12}]_1=0\\
 [E^\phi_{1},E^\phi_{12}]_q&=[E^\phi_{1},F^\phi_{12}]_q=0\label{E1E12}\\
[F^\phi_{1},E^\phi_{12}]_{q^{-1}}&=[F^\phi_{1},F^\phi_{12}]_{q^{-1}}=0\label{E1F12}\\
[E^\phi_{1},F^\phi_{1}]_{q^2}&=\frac{q^2}{q-q^{-1}}1\label{esuntereinander}\\
[E^\phi_{12},F^\phi_{12}]_{q^2}&=F^\phi_{1}E^\phi_{1}(q^4-q^2)+\frac{q^4}{q-q^{-1}}1\label{E12F12}
\end{align}

We prove that $C$ is basic: The algebra $C$ contains a subalgebra $\langle
E^\phi_1,F^\phi_1\rangle$ isomorphic to the Weyl algebra, which acts by
Example \ref{Weyl algebra} on every irreducible finite-dimensional
representation $V_{e,f}$ via its quotient algebra
$E^\phi_1F^\phi_1=F^\phi_1E^\phi_1=\frac{q^2}{(q-q^{-1})(1-q^2)}$, and
these are all one-dimensional. But the relation reduces the commutator
relation \ref{E12F12} to

\begin{align}
[E^\phi_{12},F^\phi_{12}]_{q^2}&=0 
\qquad \text{on any }V_{e,f} \label{E12F12Null} 
\end{align}

Consider now for any irreducible finite-dimensional $C$-representation $V$.
Since $K$ is central it acts by a scalar. Let $V_{e,f}$ be an irreducible Weyl
algebra representation, which is one-dimensional. By the $q$-commutators
\ref{E1E12} and \ref{E1F12} the generators $E^\phi_{12},F^\phi_{12}$ map
$V_{e,f}$ to some subrepresentation $V_{qe,q^{-1}f}$ resp. $V_{q^{-1}e,qf}$.
Since $ef\neq 0$ and $q$ is not a root of unity, finite dimension proves that
$E^\phi_{12},F^\phi_{12}$ act nilpotently. Let $V_{e',f'}$ be such that
$E^\phi_{12}$ acts by zero, then the additional relation \ref{E12F12Null}
shows that $F^\phi_{12}$ preserves $V_{e',f'}$ as well (and it acts also by
zero). Hence any irreducible finite-dimensional $C$-representation is
one-dimensional.

\subsection{Classification Result}

From our proven direction of Conjecture B it follows that the basic right coideal
subalgebras introduced above are maximal among all \emph{triangular} basic right
coideal subalgebras. To get a more complete result, one has to use similar techniques as for $U_q(\sl_2)$ in Theorem \ref{thm_BorelKlassifikation_sl2}, and the arguments beyond
Conjectures A and B are very ad-hoc. We only quote the result here from the second author's PhD thesis:

\begin{theorem}[\cite{Vocke16} Chapter 10]
 The triangular basic right coideal subalgebras above 
 \begin{align*}
 &\acf[K_{1}^{\pm 1},K_{2}^{\pm 1}] S(\Uplus{w_0})\\
 &\Uminus{s_1}_{\phi_-}\acf[(K_1K_2^2)^{\pm 1}] S(\Uplus{w_0})_{\phi_+}\\
 &\Uminus{s_1s_2}_{\phi_-}\acf[(K_1K_2^2)^{\pm 1}] S(\Uplus{s_1s_2})_{\phi_+}
 \end{align*}
 are Borel subalgebras and these are all Borel subalgebras of $U_q(\sl_3)$ up to reflection and diagram automorphism.
\end{theorem}

The proof idea for both assertions follows from \cite[Thm. 4.11, Lm. 3.5]{Vocke16}: We have for any right coideal subalgebra in $U$ a generating
system of elements, which have a unique leading term that is a root vector. Then
one has to study combinations of arbitrary elements and test them on (minuscule)
representations of $U_q(\sl_3)$ to rule out that the algebra is basic, or to show that the combination
is contained in a triangular right coideal subalgebra because we know our
examples are the maximal ones among them. 

\subsection{Induction of one-dimensional characters}\label{sec_exampleA2}

We want to finally calculate the induced representations in the two nontrivial Borel subalgebras above in $U_q(\sl_3)$ as in Section \ref{sec_induction}.\\

Let $C=\Uminus{s_1}_{\phi_-}\acf[(K_1K_2^2)^{\pm 1}] S(\Uplus{s_1s_2s_1})_{\phi_+}$.
We have seen that each irreducible representation is of the form
$V_{e_1,f_1,k_{122}}$, which is one-dimensional with action given by scalars $e_1,f_1,k_{122}\in \acf^\times$ with $e_1f_1=\frac{q^2}{(q-q^{-1})(1-q^2)}$
\begin{align*}
F_{1}+\lambda' K_{1}^{-1}
&\mapsto f_1\\
(K_1K_2^2)^{\pm 1} 
&\mapsto k_{112}\\
E_{{1}}K_{{1}}^{-1}+\lambda K_{1}^{-1} 
&\mapsto e_1\\
E_{2}K_{2}^{-1}
&\mapsto 0 
\end{align*}
Since $C$ has a PBW-basis we can calculate the induced module to be
\begin{align*}
U_q(\sl_3) \otimes_C V_{e_1,f_1,k_{122}}
&=\langle F_2^iF_{12}^j K_1^k K_2^\epsilon \mid i,j\in\N,k\in \Z,\epsilon=0,1 \rangle_\acf\\
&=\langle F_2^iF_{12}^j K_2^k \mid i,j\in\N,k\in \Z\rangle_\acf
\end{align*}
This is (up to $K_2^\epsilon$, which depends on the chosen lattices) the right
coideal subalgebra $\tilde{C}=\Uminus{s_2s_1}\acf[K_1]$, acting on itself by left-multiplication, and this action is extended to $U_q(\sl_3)$ acting on $\tilde{C}$. A similar result follows easily whenever $\Phi^+(w_-)\cap\Phi^+(w_+)=\supp$ as we show in \cite{LV17}. \\


Let $C=\Uminus{s_1s_2}_{\phi_-}\acf[(K_1K_2^2)^{\pm 1}] S(\Uplus{s_1s_2})_{\phi_+}$.
We have seen that each irreducible representation is of the form
$V_{e_1,f_1,k_{122}}$, which is one-dimensional with action given by scalars $e_1,f_1,k_{122}\in \acf^\times$ with $e_1f_1=\frac{q^2}{(q-q^{-1})(1-q^2)}$
\begin{align*}
E_{1}K_{1}^{-1}+\lambda K_{1}^{-1}
&\mapsto e_1\\
F_{1}+\lambda' K_{1}^{-1}
&\mapsto f_1\\
(K_1K_2^2)^{\pm1}
&\mapsto k_{122}\\
E_{12}(K_1K_2)^{-1}+(1-q^{-2})\lambda E_{2}(K_1K_2)^{-1}
&\mapsto 0\\
F_{12}+(q^{-1}-q)\lambda' F_{2}K_{1}^{-1}
&\mapsto 0
\end{align*}Since $C$ has a PBW-basis, we can calculate the induced module to be
\begin{align*}
U_q(\sl_3) \otimes_C V_{e_1,f_1,k_{122}}
&=\langle F_2^iE_2^j K_1^k K_2^\epsilon \mid i,j\in\N,k\in \Z,\epsilon=0,1 \rangle_\acf\\
&=\langle F_2^iE_2^j K_2^k \mid i,j\in\N,k\in \Z\rangle_\acf
\end{align*}
This is (up to $K_2^\epsilon$, which depends on the chosen lattices) the Hopf subalgebra $\tilde{C}=U_q(\sl_2)$, acting on itself by left-multiplication, and this action is extended to $U_q(\sl_3)$ acting on $\tilde{C}$.\\

In both cases our results in Section \ref{sec_induction} state that any finite-dimensional $U_q(\sl_3)$ representation appears as quotient of the induced representation for some specific values $(e_1,f_1,k_{122})$. The tool in Lemma \ref{lm_quotientIndA1} shows again that for generic values the induced representations are irreducible.

\section{Outlook}

Our efforts are far from being concluded. We wish to point out from our perspective difficult points that need to be resolved in future research:
\begin{enumerate}
 \item Conjecture A should be proven for all $\g$. One could attempt to treat the remaining critical cases by hand, by computing explicitly the character shifts in non-multiplicity-zero cases of Criterion 3, but this would be very tedious. It would be much more satisfying to have a more systematic proof.
 \item Conjecture B should be proven in the open direction, and the graded algebra $\gr(C)$ together with Conjecture A should be used more systematically. 
 \item Similarly, we would expect that the graded algebra $\gr(C)$ gives much information about the structure of the induced module. In essence, one would want to use the preimage of $\gr(C)^0$ much like the Cartan part for a usual Verma module. One should also try to prove standard facts, for example, that the tops are irreducible modules.
 \item In \cite{LV17} Section 5 we study all triangular basic coideal subalgebras that are maximal among the triangular coideal subalgebras --- conjecturally they are Borel. We find extensions of quantum Weyl algebras by quantum Weyl algebras in a way controlled by families of Weyl group elements, and the first stage we can construct uniformly for $A_n$. Are there completely different types of examples?
 \item Beyond triangular Borel subalgebras, we have observed in small examples (but do not dare to conjecture) that for a possibly non-triangular basic right coideal subalgebras, there always seems to be a larger triangular right coideal subalgebras, which is still basic. This would imply all Borel subalgebras 
 are triangular. If not, then a counter-example would be extremely interesting. Up to now, the only access we have is to use the generator theorem in \cite{Vocke16} for arbitrary coideal subalgebras and then combine possible generators by hand. 
  \item The case of $q$ a root of unity is more difficult and very interesting. For example, the restriction in Example \ref{exm_restriction} is irreducible for $q^2=-1$, hence $B_{\lambda,\lambda'}$ is not basic, similarly for the projective simple module for other $q$ of finite order. It is interesting that these are precisely the negligible modules in the sense of semisimplification, so one might attempt to consider $B_{\lambda,\lambda'}$ still to be basic in such a more complicated sense.
\end{enumerate}
Moreover, the following questions are from our perspective interesting with respect to applications:
\begin{enumerate}
 \item For any Borel subalgebra $B$, do the induced modules produce a good analog of the tensor category $\mathcal{O}_B$? How does the fact that $B$ is not a Hopf subalgebra, but still a coideal subalgebra, influence such a theory? For quantum symmetric pairs, there have been interesting results in \cite{Kolb19}.
 \item The induced modules appear for $\sl_2$ in \cite{Schm96,Tesch01} in the context of non-compact quantum group. For $\g$, are there corresponding families of modules with these additional analytic data? (for example, for $\sl_2$ replacing $\C[K,K^{-1}]$ by functions on the torus).  On the other hand, is our construction related to the non-standard Borel subalgebras of affine Lie algebras \cite{Fut94, Cox94} via Kazhdan-Lusztig correspondence?
 \item Our non-standard Borel subalgebras suggest non-standard free field realizations in logarithmic conformal field theory \cite[Rem. 2.11]{CLR23}. 
\end{enumerate}

\enlargethispage{1cm}


\begin{thebibliography}{xxxxxx}

\bibitem[AHS10]{AHS10} N. Andruskiewitsch, I. Heckenberger, H.-J. Schneider:
 \emph{The Nichols algebra of a semisimple {Y}etter-{D}rinfeld module},
 Amer. J. Math. 132 (2010) 1493-1547
\bibitem[Beck16]{Beck16} Stefan Beck: \emph{A Family of Right Coideal Subalgebras of $U_q(\sl_{n+1})$}, Dissertation Philips University Marburg (2016), Advisor: I. Heckenberger, https://archiv.ub.uni-marburg.de/diss/z2016/0217/pdf/dsb.pdf.

\bibitem[Cox94]{Cox94} B. Cox: \emph{Verma modules induced from non-standard Borel algebras}, Pacific Journal of mathematics 165/2 (1994), 269-2434. 

\bibitem[CLR23]{CLR23} T. Creutzig, S. Lentner, M. Rupert: \emph{An algebraic theory for logarithmic Kazhdan-Lusztig correspondences}, Preprint (2023),  arXiv:2306.11492.

\bibitem[Fut94]{Fut94} V. M. Futorny: \emph{Imaginary Verma modules for affine Lie algebras}, Canad. Math. Bull. Vol. 37/2 (1994), 213-218. 

\bibitem[Jan96]{Jan96}
J.~Jantzen,
\emph{Lectures in Quantum Groups},
Graduate Studies in Mathematics, Volume 6, American Mathematical Society, 1996.



\bibitem[HK11a]{HK11a}
I.~Heckenberger and S.~Kolb, 
\emph{Homogeneous right coideal subalgebras of quantized enveloping
 algebras}, Bull. London Math. Soc. 44, 837-848, 2012.


\bibitem[HK11b]{HK11b}
I.~Heckenberger and S.~Kolb, 
\emph{Right coideal subalgebras of the {B}orel part of a quantized
 enveloping algebra}, Int. Math. Res. Not. IMRN, no.~2, pp
 419-451, 2011.
 
\bibitem[HS09]{HS09}
I.~Heckenberger and H.-J. Schneider, 
\emph{Right coideal subalgebras of
 {N}ichols algebras and the {D}uflo order on the {W}eyl groupoid}, Israel Journal of Mathematics 197, 139-187, 2013.

 \bibitem[Kolb14]{Kolb14} S. Kolb: \emph{Quantum symmetric Kac-Moody pairs}, Adv. Math. 267 (2014), 395-469.
\bibitem[Kolb19]{Kolb19} S. Kolb: \emph{Braided module categories via quantum symmetric pairs}, Proceedings of the London mathematical society 21/1 (2020) p. 1-31.
 \bibitem[KS08]{KS08} V.K. Kharchenko and A.V. Lara Sagahon: \emph{Right Coideal Subalgebras in $U_q(\sl_{n+1})$}, 
Journal of Algebra 319 (2008), pp. 2571–2625

\bibitem[Let97]{Let97} G. Letzter: \emph{Subalgebras which appear in quantum Iwasawa decompositions}, Canadian Journal of Mathematics 49 (1997), 1206--1223.
\bibitem[Let99]{Let99} G. Letzter: \emph{Symmetric Pairs for Quantized Enveloping Algebras}, Journal of Algebra 220 (1999), 729-767.
\bibitem[LMO88]{LMO88} A. Leroy, J. Matczuk, J. Okninski: \emph{On the Gelfand-Kirillov dimension of normal localizations and twisted polynomial rings}, Perspectives in ring theory (1988) 205-214.
\bibitem[LP17]{LP17} S. Lentner, J. Priel: \emph{Three natural subgroups of the Brauer-Picard group of a Hopf algebra with applications},
Bull. Belg. Math. Soc. Simon Stevin 24 (2017), p. 73-106
\bibitem[Lusz90]{Lusz90} G. Lusztig: \emph{Finite dimensional Hopf algebras
 arising from quantized universal enveloping algebras}, J. Amer. Math. Soc. 3
 (1990), 257-296. 
\bibitem[Lusz93]{Lusz93} G. Lusztig: \emph{Introduction to Quantum Groups}, Modern Birkhäuser Classics (2010, first appearance 1993), Birkhäuser Basel. 
\bibitem[LV17]{LV17} S. Lentner, K. Vocke: \emph{Constructing new Borel subalgebras of quantum groups with a non-degeneracy property}, Preprint (2017), arXiv:1702.06223
\bibitem[M64]{M64} H. Matsumoto: \emph{G\'en\'erateurs et relations des groupes de weyl g\'en\'eralis\'es}, C.R. Acad. Sci. Paris (1964) 258.
\bibitem[MS89]{MS89} J. C. McConnell, J. T. Stafford: \emph{Gelfand-Kirillov dimension and associated graded modules},
Journal of Algebra, Volume 125, Issue 1 (1989), 197-214.
\bibitem[NS95]{NS95} M. Noumi and T. Sugitani: \emph{Quantum symmetric spaces and related q-orthogonal polynomials}, Group theoretical methods in physics, World Scientific (1995), 28-40.
\bibitem[Schm96]{Schm96} K. Schmüdgen: \emph{Operator representations of $U_q(\sl(2,\mathbb{R}))$}, Lett. Math. Phys. 37 (1996) 211-222.
\bibitem[Skry06]{Skry06} S. Skryabin: \emph{Projectivity and Freeness over Comodule Algebras}, Transactions of the American Mathematical Society
359/6 (2007), p. 2597-2623. 
\bibitem[Sten75]{Sten75} B. Stenström: \emph{Rings of quotients - An introduction to methods of ring theory}, 217, Springer (1975).
\bibitem[Tesch01]{Tesch01} J. Teschner: \emph{Liouville theory revisited}, Preprint (2001), arXiv:hep-th/0104158.
\bibitem[Vocke16]{Vocke16} K. Vocke: \emph{\"Uber Borelunteralgebren
von Quantengruppen}, Dissertation (2016), Phillips-University Marburg, \texttt{https://archiv.ub.uni-marburg.de/diss/z2016/0660/}.

\end{thebibliography}
\end{document}